\documentclass[12pt,twoside]{amsart}
\usepackage{amsfonts,amsmath,amstext,amsbsy,amsthm,amscd,amsmath,amssymb,dsfont}
\usepackage[utf8]{inputenc}
\usepackage{hyperref}
\usepackage[T1]{fontenc}
\usepackage{color}

\newcommand{\norme}[1]{\left\Vert #1\right\Vert}
\newtheorem{Lemma}{Lemma}[section]
\newtheorem{Prop}{Proposition}[section]  

\newtheorem{Rmk}{Remark}[section]
\newtheorem{Thm}{Theorem}[section]

\theoremstyle{remark}

\usepackage{indentfirst}
\newcommand{\be}{\begin{equation}}
\newcommand{\ee}{\end{equation}}
\newcommand{\ba}{\begin{array}}
\newcommand{\ea}{\end{array}}
\newcommand{\bea}{\begin{eqnarray}}
\newcommand{\eea}{\end{eqnarray}}
\newcommand{\bee}{\begin{eqnarray*}}
\newcommand{\eee}{\end{eqnarray*}}

\renewcommand{\div}{\mbox{\rm div}\;\!}

\newcommand{\R} {\mathbb{R}}

\newcommand{\cC} {\mathcal{C}}

\def\var{\varepsilon}

\def\dZ_1{\delta\!Z_1}

\def\d{\partial}

\def\ddj{\dot\Delta_j}

\title[Relaxation limit of the Jin-Xin system]{Diffusive relaxation limit of the multi-dimensional hyperbolic Jin-Xin system}

\author[T. Crin-Barat]{Timothée Crin-Barat${^*}$}
\address[T. Crin-Barat]{Chair for Dynamics,
Control, and Numerics (Alexander von Humboldt
Professorship), Department of Data
Science, Friedrich-Alexander University Erlangen-Nuremberg, Erlangen,
Germany}
\email{timothee.crin-barat@fau.de}

\author[L-Y. Shou]{Ling-Yun Shou}
\address[L-Y. Shou]{Department of Mathematics, Nanjing University of Aeronautics and Astronautics, Nanjing, 211106, P. R. China.}
\email{shoulingyun11@gmail.com}

\subjclass[201,0]{35Q35; 76N10}
\keywords{Jin-Xin approximation, relaxation limit, global existence, Besov spaces\\\quad$^*$ Corresponding author: timotheecrinbarat@gmail.com}

\begin{document}
\maketitle
\begin{abstract}
We study the diffusive relaxation limit of the Jin-Xin system toward viscous conservation laws in the multi-dimensional setting. For initial data being small perturbations of a constant state in suitable homogeneous Besov norms, we prove the global well-posedness of strong solutions satisfying uniform estimates with respect to the relaxation parameter. Then, we justify the strong relaxation limit and exhibit an explicit convergence rate of the process. Our proof is based on an adaptation of the techniques developed in \cite{c1,c2} to be able to deal with additional low-order nonlinear terms.
\end{abstract}



\section{Introduction}
Hyperbolic relaxation systems arise in a wide variety of physical situations, ranging from  non-equilibrium gas dynamics, flood flows, water waves, kinetic theory, etc \cite{cercignani1,vicenti1,whitham1}. In the present paper, we investigate the multi-dimensional Jin-Xin relaxation system under the diffusive scaling:
\begin{equation} \label{JXD}
\left\{
    \begin{aligned}
&\frac{\partial} {\partial t}u+\sum_{i=1}^{d} \frac{\partial} {\partial x_{i}} v_{i}=0, \\
&\varepsilon^2\frac{\partial} {\partial t} v_{i}+A_{i}\frac{\partial} {\partial x_{i}} u=-\big(v_{i}-f_{i}(u) \big),\quad\quad i=1,2,...,d,
    \end{aligned}
    \right.
\end{equation}
where $\varepsilon>0$ is the relaxation parameter, $t>0$ denotes the time variable, and $x\in\mathbb{R}^{d}$ stands for the space variable. The unknowns are $u=u(x,t)\in \mathbb{R}^{n}$ and $v=(v_{1},v_{2},...,v_{d})$ with $v_{i}=v_{i}(x,t)\in \mathbb{R}^{n}$. 

We are interested in the existence of global-in-time strong solutions $(u,v)$ of \eqref{JXD} that are close to constant equilibrium $(\bar{u},\bar{v})$ with $\bar{u},\bar{v}\geq0$. The constant coefficient matrices $A_{i}$ are taken as
\begin{align}
 A_{i}=a_{i} I_{n},\quad\quad a_{i}>0,\quad\quad i=1,2,...,d,   \label{AssumptionA}
\end{align}
and the nonlinear term $f(u)=(f_{1}(u),f_{2}(u),...,f_{m}(u))$ with $f_{i}(u)\in \mathbb{R}^{n}$ is assumed to be smooth in $u$ and satisfy
\begin{align}
 f_{i}(\bar{u})=\bar{v},\quad\quad f_{i}'(\bar{u})=0,\quad\quad i=1,2,...,d,\label{Assumptionf}
\end{align}
where the first equality of \eqref{Assumptionf} ensures that $(\bar{u},\bar{v})$ is a solution of \eqref{JXD} and the second ensures that $f(u)$ is \textit{at least quadratic}.

Formally, as the relaxation parameter $\varepsilon$ tends to $0$, the solution $(u,v)$ converges to $(u^*,v^{*})$, where $u^{*}$ solves the viscous conservation laws
\begin{equation}
\frac{\partial} {\partial t}u^{*}+\sum_{i=1}^{d} \frac{\partial} {\partial x_{i}}f_{i}(u^{*})=\sum_{i=1}^{d} \frac{\partial} {\partial x_{i}} (A_{i}\frac{\partial} {\partial x_{i}} u^{*}),\label{uheat}
\end{equation}
and $v^*$ is given by
\begin{equation}
v_{i}^*=-A_{i}\frac{\partial} {\partial x_{i}} u^*+f_{i}(u^*),\quad\quad i=1,2,...,d.\label{Darcy}
\end{equation}
\indent This relaxation procedure can be interpreted as a  finite speed of propagation approximation as one directly sees that the hyperbolicity of \eqref{JXD} ensures the finite speed of propagation while the second-order term in \eqref{uheat} implies an infinite speed of propagation.
Moreover, such relaxation procedure has practical importance in numerical analysis as it reduces the constraints and therefore simplifies the computation of solutions, cf. \cite{LiuViscous} and references therein for more details. 
\medbreak
The relaxation of hyperbolic systems was first analysed by the fundamental work of Liu \cite{liu1}. Then, System \eqref{JXD} was introduced by Jin and Xin in \cite{jin1}  to approximate hyperbolic conservation laws and has been intensively studied since then. Let us now point out that there are two types of Jin-Xin approximations:

\begin{itemize}
 \item The diffusive scaling depicted above in \eqref{JXD} that is used to approximate the $d$-dimensional viscous conservation laws \eqref{uheat}(cf. \cite{jin0}).
    \item And the hyperbolic scaling that reads:
    \begin{equation} \label{JXDHyp}
\left\{
    \begin{aligned}
&\frac{\partial} {\partial t}u+\sum_{i=1}^{d} \frac{\partial} {\partial x_{i}} v_{i}=0, \\
&\frac{\partial} {\partial t} v_{i}+A_{i}\frac{\partial} {\partial x_{i}} u=-\frac{1}{\var}\big(v_{i}-f_{i}(u) \big),\quad\quad i=1,2,...,d,
    \end{aligned}
    \right.
\end{equation}
which serves as an approximation of the following $d$-dimensional $n\times n$ hyperbolic conservation law (cf. \cite{jin1}):
\begin{align}\label{CL}
\frac{\partial} {\partial t}u^{*}+\sum_{i=1}^{d} \frac{\partial} {\partial x_{i}}f_{i}(u^{*})=0.
\end{align}
  
\end{itemize}
\noindent
An explicit example of \eqref{JXDHyp} in two dimensions, for which a numerical analysis has been performed in \cite{jin1}, reads
\begin{equation} \label{3compEx}
\left\{
    \begin{aligned}
&\frac{\partial} {\partial t}u+\frac{\partial}{\partial x_1}v_1+\frac{\partial}{\partial x_2}v_2=0
\\&\frac{\partial} {\partial t}v_1+A_1\frac{\partial}{\partial x_1}u=-\frac{1}{\var}\big(v_1-f_1(u) \big),
\\&\frac{\partial} {\partial t}v_2+A_2\frac{\partial}{\partial x_2}u=-\frac{1}{\var}\big(v_2-f_2(u) \big).
    \end{aligned}
    \right.
\end{equation}
\bigbreak

In the present paper, we focus on the diffusive scaled version of Jin-Xin system \eqref{JXD}. One can expect the hyperbolic scaled version to be treatable in a similar fashion but only for local-in-time solution. This comes from the fact that solutions of \eqref{CL} supplemented with 
smooth data  might develop singularities (shock waves) in finite time even if the initial data are small  perturbations of an equilibrium state (for instance, the works by Majda in \cite{Majda} and Serre in \cite{Serre}).

\vspace{2ex}

The one-dimensional Jin-Xin system is well understood, refer to the non-exhaustive literature \cite{bianchini0,bianchini1,chern1,jin0,jin2,liuh1,liuh2,mascia1,natalni1,orive1,serre1} and the references therein. For $\var=1$, Chern \cite{chern1} studied the time-asymptotical stability of the diffusive waves for the Jin-Xin system in the real line by virtue of the Chapman-Enskog expansion, Mei and Rubino \cite{mei1} got the both algebraic and exponential time-convergence rates of solutions to the Jin-Xin system toward the traveling wave solution in the half line, and Orive and Zuazua \cite{orive1} obtained the algebraic time-decay rates of solutions to the Jin-Xin system in the real line by reformulating the equation in a damped wave equation in the case $f'(0)=0$ and further investigated the finite-time blow-up phenomenon. Under the hyperbolic scaling, Natalini \cite{natalni1} proved the relaxation limit of the Jin-Xin system in any finite time-interval. Under the diffusive scaling, Jin and Liu \cite{jin0} justified the diffusive relaxation limit of the Jin-Xin system for initial data around a traveling wave solution, and Bouchut, Guarguaglini and Natalini \cite{bouchut1} considered the diffusive relaxation process of the Jin-Xin model in terms of BGK type approximations. Recently, Bianchini \cite{bianchini1} derived the sharp time-decay estimates of solutions to the Jin-Xin system in the case $f'(0)\ne0$ which are uniform in the relaxation parameter and provides the convergence to a nonlinear heat equation both asymptotically in time and in the relaxation limit. For complete reviews on the Jin-Xin system, refer to \cite{mascia1,natalni2}.


However, to the best of our knowledge, the relaxation limit of the multi-dimensional Jin-Xin system \eqref{JXD} has not yet been justified despite its relevance. The purpose of the present paper is to provide such justification in the framework of homogeneous Besov spaces. Let us stress here that the use of such a functional framework, somewhat complex, is not only to obtain optimal results in terms of regularity but it also allows to reflect perfectly the frequency-dependent behaviors of the system and appears to be very handy to justify the relaxation limit. In particular, a similar approach in the classical Sobolev framework does not lead to explicit convergence rates.


\vspace{1ex}

To be more precise, we first show a global-in-time well-posedness result to the Cauchy problem for System \eqref{JXD} for small initial perturbation and establish qualitative regularity properties and uniform-in-$\var$ bounds on the solution. Then, we show  that the solution of System \eqref{JXD} converges strongly, as the relaxation parameter $\var\rightarrow0$, to the solution of \eqref{uheat}-\eqref{Darcy} globally in time and we exhibit an explicit convergence rate.

\vspace{1ex}



To achieve our results, the first step is to establish global a-priori estimates which are uniform with respect to the relaxation parameter $\var$. This requires three main ingredients: Firstly, we exhibit a damped mode which has better decay properties than the whole solution and can be used to decouple the system into a purely damped equation and a parabolic one, in low frequencies. 
Secondly, we construct a Lyapunov functional, in the spirit of Beauchard and Zuazua's paper \cite{BZ} and Danchin's \cite{danchin1}, that encodes enough information to recover all the dissipative properties in high frequencies. Finally, we choose a suitable threshold between the low and high frequencies to be able to handle the overdamping phenomena (cf. Remark \ref{remarkoverdamping}) and close the a-priori estimates with enough uniformity. These ingredients are similar to the one present in the works \cite{CBD1,CBD2} concerning general partially dissipative hyperbolic systems and in particular the compressible Euler equations with damping.

\smallbreak
Yet, new difficulties arise in our analysis due to the low-order nonlinear term $f(u)$ that can not be handled in the framework of \cite{CBD1,CBD2,ThesisCB}.
More precisely, due to this nonlinear term:
\begin{itemize}
    \item We can not perform a rescaling as in \cite{CBD1,CBD2} so as to reduce the proofs to the case $\var=1$ and then recover the exact dependency with respect to the parameter $\var$ by scaling back.
    \item It is difficult to estimate the low-order nonlinearity $f(u)$ as a source term in the low-frequency regime as it is not a function of the directly damped component $v$. This obliges us to work in a different low-frequency regularity setting.
    \item Moreover, even if the (SK) condition developed by Kawashima and Shizuta in \cite{kaw1,kaw2} is satisfied by the linearized version of System \eqref{JXD} (i.e. without $f$), it does not allow us to prove the global existence uniformly in $\varepsilon$ as the relaxation parameter plays a key role to deal with the nonlinear term $f(u)$.
\end{itemize}

To overcome these difficulties, delicate energy estimates uniform-in-$\var$ are needed for both low and high frequencies in the spirit of \cite{CBD1,CBD2,c2}. Compared to these references, the damped mode used to diagonalize the system in low frequencies will involve the nonlinear term $f(u)$, and additional estimates have to be established in $\dot{B}^{\frac{d}{2}-1}_{2,1}$.

\smallbreak

The paper is organized as follows. In Section \ref{section2}, we state our main results in the multi-dimensional case $d\geq2$. In Section \ref{section3}, we establish the uniform a-priori estimates and prove the global well-posedness of the Cauchy problem for the Jin-Xin model. The rigorous justification of the relaxation limit is performed in Section \ref{section4}. Section \ref{section5} is devoted to the one-dimensional case $d=1$. Some technical results are recalled in the Appendix.

\medbreak

\section{Main results}\label{section2}
\subsection{Littlewood-Paley notations}

Before stating our main results, we recall the notations of the Littlewood-Paley decomposition and Besov spaces. The reader can refer to \cite{bahouri1}[Chapter 2] for a complete overview. Choose a smooth radial non-increasing function $\chi(\xi)$ with compact supported in $B(0,\dfrac{4}{3})$ and $\chi(\xi)=1$ in $B(0,\dfrac{3}{4})$ such that
$$
\varphi(\xi):=\chi(\frac{\xi}{2})-\chi(\xi),\quad \sum_{j\in \mathbb{Z}}\varphi(2^{-j}\cdot)=1,\quad \text{{\rm{Supp}}}~ \varphi\subset \{\xi\in \mathbb{R}^{d}~|~\frac{3}{4}\leq |\xi|\leq \frac{8}{3}\}.
$$
For any $j\in \mathbb{Z}$, the homogeneous dyadic blocks $\dot{\Delta}_{j}$ and the low-frequency cut-off operator $\dot{S}_{j}$ are defined by
$$
\dot{\Delta}_{j}u:=\mathcal{F}^{-1}(\varphi(2^{-j}\cdot )\mathcal{F}u),\quad\quad \dot{S}_{j}u:= \mathcal{F}^{-1}( \chi (2^{-j}\cdot) \mathcal{F} u),
$$
where $\mathcal{F}$ and $\mathcal{F}^{-1}$ stand for the Fourier transform and its inverse. 
Let $\mathcal{S}_{h}'$ be the set of tempered distributions on $\mathbb{R}^{d}$ such that every $u\in \mathcal{S}_{h}'$ satisfies $u\in \mathcal{S}'$ and $\lim_{j\rightarrow-\infty}\|\dot{S}_{j}u\|_{L^{\infty}}=0$. Then one has
\begin{equation}\nonumber
\begin{split}
&u=\sum_{j\in \mathbb{Z}}u_{j}\quad\text{in}~\mathcal{S}',\quad\quad \dot{S}_{j}u= \sum_{j'\leq j-1}u_{j'},\quad \forall u\in \mathcal{S}_{h}'.
\end{split}
\end{equation}
With the help of these dyadic blocks, the homogeneous Besov space $\dot{B}^{s}_{p,r}$ is defined by
$$
\dot{B}^{s}_{p,r}:=\{u\in \mathcal{S}_{h}'~|~\|u\|_{\dot{B}^{s}_{p,r}}:=\Big( \sum_{j\in\mathbb{Z} } \big( 2^{js}\|\ddj u\|_{L^{p}} \big)^{r} \Big)^{\frac{1}{r}}<\infty\}.
$$
We denote the Chemin-Lerner type space $\widetilde{L}^{\varrho}(0,T;\dot{B}^{s}_{p,r})$ as follows:
$$
\widetilde{L}^{\varrho}(0,T;\dot{B}^{s}_{p,r}):=\{u\in L^{\varrho}(0,T;\mathcal{S}'_{h})~|~ \|u\|_{\widetilde{L}^{\varrho}_{T}(\dot{B}^{s}_{p,r})}:=\Big( \sum_{j\in\mathbb{Z} } \big( 2^{js}\|\ddj u\|_{L^{\varrho}_{t}(L^{p})} \big)^{r} \Big)^{\frac{1}{r}}<\infty\}.
$$
Moreover, we write
\begin{equation}\nonumber
\begin{split}
&\mathcal{C}_{b}(\mathbb{R}_{+};\dot{B}^{s}_{p,r}):=\{u\in\mathcal{C}(\mathbb{R}_{+};\dot{B}^{s}_{p,r})~|~\|f\|_{\widetilde{L}^{\infty}(\mathbb{R}_{+};\dot{B}^{s}_{p,r})}<\infty\}.
\end{split}
\end{equation}

In order to restrict our Besov norms to low and high frequencies, we set an integer $J_{\var}$, called threshold, to be chosen later. We use the following notations:
\begin{equation}\nonumber
\begin{split}
&\|u\|_{\dot{B}^{s}_{p,r}}^{h}:=\Big( \sum_{j\leq J_{\var}} \big( 2^{js}\|\ddj u\|_{L^{p}} \big)^{r} \Big)^{\frac{1}{r}},\quad\quad\quad\quad \|u\|_{\dot{B}^{s}_{p,r}}^{h}:=\Big( \sum_{j\geq J_{\var}-1} \big( 2^{js}\|\ddj u\|_{L^{p}} \big)^{r} \Big)^{\frac{1}{r}},\\
&\|u\|_{\widetilde{L}^{\varrho}_{T}(\dot{B}^{s}_{p,r})}^{h}:=\Big( \sum_{j\leq J_{\var}} \big( 2^{js}\|\ddj u\|_{L^{\varrho}_{T}(L^{p})} \big)^{r} \Big)^{\frac{1}{r}}, \quad \|u\|_{\widetilde{L}^{\varrho}_{T}(\dot{B}^{s}_{p,r})}^{h}:=\Big( \sum_{j\geq J_{\var}-1} \big( 2^{js}\|\ddj u\|_{L^{\varrho}_{T}(L^{p})} \big)^{r} \Big)^{\frac{1}{r}}.
\end{split}
\end{equation}
For any $u\in\mathcal{S}'_{h}$, we also define the low-frequency part $u^{\ell}$ and the high-frequency part $u^{h}$ by
$$
u^{\ell}:=\sum_{j\leq J_{\var}-1}u_{j},\quad\quad u^{h}:=u-u^{\ell}=\sum_{j\geq J_{\var}}u_{j}.
$$
It is easy to check for any $s'>0$ that
\begin{equation}\label{lhl}
\begin{aligned}
&\|u^{\ell}\|_{\dot{B}^{s}_{2,1}}\leq \|u\|_{\dot{B}^{s}_{2,1}}^{\ell}\leq 2^{J_{\var}s'}\|u\|_{\dot{B}^{s-s'}_{2,1}}^{\ell},\quad\quad\|u^{h}\|_{\dot{B}_{2,1}^{s}}\leq \|u\|_{\dot{B}_{2,1}^{s}}^{h}\leq 2^{-(J_{\var}-1)s'}\|u\|_{\dot{B}_{2,1}^{s+s'}}^{h}.
\end{aligned}
\end{equation}

\subsection{Main results}

First, we state a uniform-in-$\var$ existence result for System \eqref{JXD}.

\begin{Thm}\label{Thm1} Let $d\geq2$, $n\geq1$, $\bar{u},\bar{v}\geq0$, $\varepsilon\in(0,1)$ and assume \eqref{AssumptionA}-\eqref{Assumptionf}. Set the threshold between the low and high frequencies 
\begin{align}
&J_{\var}:=-[\log_{2}\var]-k_{0},\label{Jvar}
\end{align}
with $k_{0}>0$ a suitably large constant independent of $\var$. There exists a constant $\eta_0>0$ independent of $\var$ such that for initial data $(u_{0},v_{0})$ satisfying $(u_{0}-\bar{u},v_{0}-\bar{v})\in\dot{B}^{\frac{d}{2}-1}_{2,1}\cap\dot{B}^{\frac{d}{2}}_{2,1}$ and
 \begin{equation}\label{small}
\|u_{0}-\bar{u}\|_{\dot{B}^{\frac{d}{2}-1}_{2,1}\cap\dot{B}^{\frac{d}{2}}_{2,1}}+ \var \|v_{0}-\bar{v}\|_{\dot{B}^{\frac{d}{2}-1}_{2,1}\cap\dot{B}^{\frac{d}{2}}_{2,1}}\leq \eta_0,
\end{equation} 
then System \eqref{JXD} associated to the initial data $(u_{0},v_{0})$ admits a unique global strong solution $(u,v)$ satisfying $(u-\bar{u},v-\bar{v})\in \mathcal{C}_{b}(\mathbb{R}_{+};\dot{B}^{\frac{d}{2}-1}_{2,1}\cap\dot{B}^{\frac{d}{2}}_{2,1})$ and
\begin{equation}\label{Xt}
\begin{aligned}
&\|u-\bar{u}\|_{\widetilde{L}^{\infty}_{t}(\dot{B}^{\frac{d}{2}-1}_{2,1}\cap\dot{B}^{\frac{d}{2}}_{2,1})}+\var\| v-\bar{v}\|_{\widetilde{L}^{\infty}_{t}(\dot{B}^{\frac{d}{2}-1}_{2,1}\cap\dot{B}^{\frac{d}{2}}_{2,1})}\\
&\quad\quad+\|u-\bar{u}\|_{L^1_{t}(\dot{B}^{\frac{d}{2}+1}_{2,1}\cap\dot{B}^{\frac{d}{2}+2}_{2,1})}^{\ell}+\frac{1}{\var^2}\|u-\bar{u}\|_{L^1_{t}(\dot{B}^{\frac{d}{2}}_{2,1})}^{h}\\
&\quad\quad+\|v-\bar{v}\|_{L^1_{t}(\dot{B}^{\frac{d}{2}}_{2,1}\cap \dot{B}^{\frac{d}{2}+1}_{2,1})}^{\ell}+\frac{1}{\var}\|v-\bar{v}\|_{L^1_{t}(\dot{B}^{\frac{d}{2}}_{2,1})}^{h}\\
&\quad\quad+\frac{1}{\var}\sum_{i=1}^{d} \|Z_i\|_{L^1_{t}(\dot{B}^{\frac{d}{2}-1}_{2,1}\cap\dot{B}^{\frac{d}{2}}_{2,1})}^{\ell}+\frac{1}{\var^2}\sum_{i=1}^{d}\|Z_i\|_{L^1_{t}(\dot{B}^{\frac{d}{2}-1}_{2,1})}^{h}\\
   &\quad\leq C(\|u_{0}-\bar{u}\|_{\dot{B}^{\frac{d}{2}-1}_{2,1}\cap\dot{B}^{\frac{d}{2}}_{2,1}}+ \var \|v_{0}-\bar{v}\|_{\dot{B}^{\frac{d}{2}-1}_{2,1}\cap\dot{B}^{\frac{d}{2}}_{2,1}})\quad\hbox{for all } t\geq 0,
\end{aligned}
\end{equation}
where $C>0$ is a constant independent of $\var$ and time, and $Z_i:=A_{i}\dfrac{\partial} {\partial x_{i}} u+v_{i}-f_{i}(u)$.
 \end{Thm}
 
 
 \begin{Rmk}
 Some comments are in order.
 \begin{itemize}
     \item By taking the threshold $J_{\var}$ in \eqref{Jvar}, we capture subtle regularity properties of $(u,\var v)$ in the low-frequency region $|\xi|\leq \frac{1}{\var}2^{-k_{0}}$ and high-frequency region $|\xi|\geq \frac{1}{\var}2^{-k_{0}}$ separately. 
     \item The definition \eqref{Jvar} of the threshold $J_\varepsilon$ implies that as $\var\rightarrow0$, the low-frequency region covers all the frequency space and the high-frequency region disappears. The low-frequency regularity estimates in \eqref{Xt} are of parabolic type, which is consistent with that of the limiting parabolic system \eqref{uheat}.
     \item The effective unknowns $Z_i$ $(i=1,2,...,d)$ is related to the equations \eqref{Darcy} verified by the limit system. They have better time-integrability and faster decay rate in $\var$ compared with the solution $(u,v)$. These properties are crucial to close the uniform a-priori estimates and also to justify the relaxation limit.
    \item As there are no advection term in \eqref{JXD}, the spatial Lipschitz regularity of the solution is not required there. Due to the \textit{at-least}-quadratic non linearity $f$, one must assume that the initial data are in $\dot{B}^{\frac{d}{2}-1}_{2,1}\cap\dot{B}^{\frac{d}{2}}_{2,1}$ to ensure the $L^\infty$ norm of the solution so as to avoid Riccati type blow-up.
     \item To include advection terms in our analysis, one should assume $(u_0^h,v_0^h)\in \dot{B}^{\frac d2+1}_{2,1}$ and derive high-frequency a priori estimates in such space as in \cite{CBD2,c1,c2}.
 \end{itemize}
 \end{Rmk}

In order to study the relaxation limit of System \eqref{JXD}, we also need to justify the existence of solution to the limit system \eqref{uheat}. This is done in the following theorem.
\begin{Thm}\label{Thm2} 
Let $d\geq2$, $n\geq1$, and $\bar{u},\bar{v}\geq0$,. There exists a constant $\eta_1>0$ such that if $u_0^*$ satisfies $u_0^*-\bar{u}\in \dot{B}^{\frac d2-1}_{2,1}\cap\dot{B}^{\frac{d}{2}}_{2,1}$ and
\begin{equation}\label{smalllimit}
\|u_{0}^*-\bar{u}\|_{\dot{B}^{\frac d2-1}_{2,1}\cap\dot{B}^{\frac{d}{2}}_{2,1}}\leq \eta_{1},
\end{equation}
then System \eqref{uheat} associated to the initial data  $u^*_0$
admits a unique global strong solution $u^*$ satisfying $u^*-\bar{u}\in \cC_b(\R^+;\dot{B}^{\frac d2-1}_{2,1}\cap\dot{B}^{\frac{d}{2}}_{2,1})$. Let  $v^{*}$ be given by \eqref{Darcy}. Then it holds that
\begin{equation}\label{Xtlimit}
    \begin{aligned}
    &\|u^*-\bar{u}\|_{L^{\infty}_{t}(\dot{B}^{\frac d2-1}_{2,1}\cap\dot{B}^{\frac{d}{2}}_{2,1})}+\|u^*-\bar{u}\|_{L^1_{t}(\dot{B}^{\frac{d}{2}+1}_{2,1}\cap\dot{B}^{\frac{d}{2}+2}_{2,1})}+\|v^*-\bar{v}\|_{L^1_{t}(\dot{B}^{\frac{d}{2}}_{2,1}\cap\dot{B}^{\frac{d}{2}+1}_{2,1})}\\
    &\quad \leq C\|u^*_0-\bar{u}\|_{\dot{B}^{\frac d2-1}_{2,1}\cap\dot{B}^{\frac{d}{2}}_{2,1}} \quad\quad\hbox{for all } t\geq 0,
    \end{aligned}
\end{equation}
where $C>0$ is a constant independent of time.
   \end{Thm}
\vspace{1ex}
   
Finally, we state a result pertaining to the relaxation limit of System \eqref{JXD}.
\begin{Thm}\label{Thm3} 
For any $d\geq2$, $n\geq1$, $\bar{u},\bar{v}\geq0$ and $\varepsilon\in(0,1)$, let $(u,v)$ be the solution of System \eqref{JXD} associated to the initial data $(u_{0},v_{0})$ given by Theorem \ref{Thm1},  $u^*$ be the global solution of System \eqref{uheat} associated to the initial data $u_{0}^{*}$ given by Theorem \ref{Thm2} and $v^{*}$ be defined by \eqref{Darcy}. Suppose further that
\begin{align}\label{error}
    \|u_0-u_0^*\|_{B^{\frac d2-2}_{2,1}\cap\dot{B}^{\frac{d}{2}-1}_{2,1}}\leq\varepsilon.
\end{align}
Then for all $t\geq0$, the following convergence estimates hold{\rm:}
\begin{equation}\label{rate}
\begin{aligned}
&\|u-u^*\|_{L^{\infty}_{t}(\dot{B}^{\frac{d}{2}-2}_{2,1}\cap\dot{B}^{\frac{d}{2}-1}_{2,1})}+\|u-u^{*}\|_{ L^1_{t}(\dot{B}^{\frac{d}{2}}_{2,1})\cap \widetilde{L}^2_{t}( \dot{B}^{\frac{d}{2}}_{2,1}) }+\|v-v^{*}\|_{L^1_{t}(\dot{B}^{\frac{d}{2}-1}_{2,1})}\leq C\varepsilon,
\end{aligned}
\end{equation}
where $C>0$ is a constant independent of time and $\var$.

Therefore, as $\var\rightarrow0$, $(u,v)$ converges to $(u^{*},v^{*})$ in the following sense{\rm:}
\begin{equation}\nonumber
\left\{
\begin{aligned}
&u\rightarrow u^{*}\quad \text{strongly in}\quad L^{\infty}(\mathbb{R}_{+};\dot{B}^{\frac{d}{2}-2}_{2,1}\cap\dot{B}^{\frac{d}{2}-1}_{2,1})\cap L^p(\mathbb{R}_{+};\dot{B}^{\frac{d}{2}}_{2,1}) ,\quad p\in[1,2],\\
&v\rightarrow v^{*}\quad\text{strongly in}\quad L^{1}(\mathbb{R}_{+};\dot{B}^{\frac{d}{2}-1}_{2,1}).
\end{aligned}
\right.
\end{equation}
\end{Thm}

\vspace{1ex}

\begin{Rmk}
By \eqref{Xt}, \eqref{rate} and interpolation, the higher order convergence estimates can be obtained as follows{\rm:}
\begin{equation}\nonumber
\begin{aligned}
\|u-u^*\|_{L^{\infty}_{t}(\dot{B}^{\frac{d}{2}-\sigma}_{2,1})}+\|v-v^{*}\|_{L^1_{t}(\dot{B}^{\frac{d}{2}-\sigma}_{2,1})}&\lesssim \varepsilon^{\sigma},\quad 0< \sigma<1.
\end{aligned}
\end{equation}
\end{Rmk}

\begin{Rmk}
Concerning the one-dimensional $(d=1)$ case, we are able to obtain the global existence and convergence estimates in some larger Besov spaces {\rm(}see Theorem \ref{Thm1d} in Section {\rm\ref{section5})}.
\end{Rmk}





\subsection{Spectral analysis of \eqref{JXD}}

Define the perturbation 
$$
m:=u-\bar{u},\quad\quad w:=\var(v-\bar{v}).
$$
Since \eqref{Assumptionf} is satisfied,  the Cauchy problem of System \eqref{JXD} subject to the initial data $(u_{0},v_{0})$ can be rewritten as
\begin{equation}\label{rJXD}
\left\{
\begin{aligned}
&\frac{\partial} {\partial t}m+\frac{1}{\var}\sum_{i=1}^{d} \frac{\partial} {\partial x_{i}} w_{i}=0, \\
&\frac{\partial} {\partial t} w_{i}+\frac{1}{\var} A_{i}\frac{\partial} {\partial x_{i}} m+\frac{1}{\var^2}w_{i}= \frac{1}{\var} (f_{i}(\bar{u}+m)-f_{i}(\bar{u}) ),\quad\quad i=1,2,...,d,\\
&(m,w)|_{t=0}=(m_{0},w_{0}):=(u_{0}-\bar{u},\var(v_{0}-\bar{v})).
\end{aligned}
\right.
\end{equation}
In order to understand the behavior of the solution and explain the choice of the threshold $J_{\var}$ in  \eqref{Jvar}, we analyse the eigenvalues of the linearized system as follows. 

For the sake of clarity, we consider the case $(d,n)=(3,1)$ and rewrite the linearized equations of \eqref{rJXD} as
 \begin{equation}\nonumber
\begin{aligned}
& \partial_{t}
\left(\begin{matrix}
   m  \\
   \var w_{1}  \\
   \var w_{2}\\
   \var w_{3}
  \end{matrix}\right)
  =\widehat{\mathbb{A}}(\xi)\left(\begin{matrix}
   m  \\
   \var w_{1}  \\
   \var w_{2}\\
   \var w_{3}
     \end{matrix}\right),\quad\quad  \widehat{\mathbb{A}}(\xi):=\left(\begin{matrix}
0                              &   -\frac{1}{\var}i\xi_{1} & -\frac{1}{\var}i\xi_{2} & -\frac{1}{\var}i\xi_{3}  &\\
-\frac{1}{\var}i a_{1} \xi_{1} & -\frac{1}{\var^2}        &  0                       & 0       &\\    
-\frac{1}{\var}i a_{2} \xi_{2} & 0                       & -\frac{1}{\var^2}        & 0       &\\
-\frac{1}{\var}i a_{3} \xi_{2} & 0                       & 0                        &  -\frac{1}{\var^2}&
  \end{matrix}\right).
   \end{aligned}
\end{equation}
Then we can compute the eigenvalues of the matrix $\widehat{\mathbb{A}}(\xi)$ from the determinant
$$
{\rm{det}}(\widehat{\mathbb{A}}(\xi)-\lambda I_{4})=(\lambda+\frac{1}{\var^2})^{2} (\lambda^2+\frac{1}{\var^2}\lambda+\frac{1}{\var^2}  \sum_{i=1}^{3} a_{i} |\xi_{i}|^2 )=0,
$$
and one has
\begin{equation}\nonumber
\left\{
\begin{aligned}
&\lambda_{1}=\lambda_{2}=-\frac{1}{\var^2},\\
&\lambda_{3}=-\frac{1}{2\var^2}+\frac{1}{2\var}\sqrt{\frac{1}{\var^2}-4\sum_{i=1}^{3} a_{i} |\xi_{i}|^2 },\\
&\lambda_{4}=-\frac{1}{2\var^2}-\frac{1}{2\var}\sqrt{\frac{1}{\var^2}-4\sum_{i=1}^{3} a_{i} |\xi_{i}|^2 }.
\end{aligned}
\right.
\end{equation}
The eigenvalues have following properties:
\begin{itemize}
\item In the low-frequency region $|\xi|\lesssim\frac{1}{\var}$, all the eigenvalues are real, and we have $\lambda_{3}\sim -|\xi|^2$ and $\lambda_{4}\sim -\frac{1}{\var^2}$.

\item In the high-frequency region $|\xi|\gtrsim\frac{1}{\var}$, the complex conjugated  eigenvalues $\lambda_{i}$ ($i=3,4$) are  equal to $-\frac{1}{2\var^2} \pm  i  \frac{ O(1) }{\var}|\xi|  $ asymptotically.
\end{itemize}
As seen above, we expect all the frequency modes to be exponentially damped uniformly in $\var$ except for the low-frequency mode associated to the eigenvalue $\lambda_3$ that undergoes a parabolic effect. This mode is related to the limit equation \eqref{uheat} and will remain after passing to the limit as $\var\to0$.
Moreover, the above spectral analysis suggests us to fix the threshold $J_{\var}\sim \frac{1}{\var}$ in order to separate the low and high behaviors of the solution. See also the remark below.


\begin{Rmk}\label{remarkoverdamping}{\rm[On the overdamping phenomenon]}
The overdamping phenomena refers to the fact that \textit{as the friction coefficient $\frac{1}{\var}$ gets larger, the decay rates do not necessarily increase and follow }$\min\{\varepsilon, \frac{1}{\varepsilon}\}$, cf. \cite{SlidesZuazua} and Remark {\rm 3.1} in \cite{burtea1} for more details. This can been seen in the above spectral analysis and prevents to deal with such relaxation limit in the Sobolev spaces setting. In the present paper, we deal with this phenomena by assuming that the threshold $J_{\var}$ satisfy $2^{J_{\var}}\sim \frac{1}{\varepsilon}$ so as to keep track of the decay rates in a suitable manner.
\end{Rmk}

\section{Proof of Theorem \ref{Thm1}}\label{section3}

\subsection{Global a-priori estimates}

\begin{Prop}\label{aprioriJX}
 For any $d\geq2$, $n\geq1$, $\var\in(0,1)$ and given time $T>0$ , suppose that $(m,w)$ is a solution to the Cauchy problem  \eqref{rJXD} such that
 \begin{equation}\label{apriori}
 \begin{aligned}
\!\! X(t):&=\|(m,w)\|_{\widetilde{L}^{\infty}_{t}(\dot{B}^{\frac{d}{2}-1}_{2,1}\cap\dot{B}^{\frac{d}{2}}_{2,1})}\\
&\quad+\|m\|_{L^1_{t}(\dot{B}^{\frac{d}{2}+1}_{2,1}\cap\dot{B}^{\frac{d}{2}+2}_{2,1})}^{\ell}+\frac{1}{\var^2}\|m\|_{L^1_{t}(\dot{B}^{\frac{d}{2}}_{2,1})}^{h}+\|m\|_{\widetilde{L}^2_{t}(\dot{B}^{\frac{d}{2}}_{2,1}\cap\dot{B}^{\frac{d}{2}+1}_{2,1})}^{\ell}+\frac{1}{\var}\|m\|_{\widetilde{L}^2_{t}(\dot{B}^{\frac{d}{2}}_{2,1})}^{h}\\
   &\quad+\frac{1}{\var}\|w\|_{L^1_{t}(\dot{B}^{\frac{d}{2}}_{2,1}\cap \dot{B}^{\frac{d}{2}+1}_{2,1})}^{\ell}+\frac{1}{\var^2}\|w\|_{L^1_{t}(\dot{B}^{\frac{d}{2}}_{2,1})}^{h}+\frac{1}{\var}\|z\|_{L^1_{t}(\dot{B}^{\frac{d}{2}-1}_{2,1}\cap\dot{B}^{\frac{d}{2}}_{2,1})}^{\ell}+\frac{1}{\var^2}\|z\|_{L^1_{t}(\dot{B}^{\frac{d}{2}-1}_{2,1})}^{h}\\
 &\leq 2C_{0}\|(m_{0},w_{0})\|_{\dot{B}^{\frac{d}{2}-1}_{2,1}\cap\dot{B}^{\frac{d}{2}}_{2,1}},\quad\quad t\in (0,T),
 \end{aligned}
 \end{equation}
where the effective unknown $z$ is defined by
\begin{equation}\label{effective}
\begin{aligned}
z_{i}:=A_{i}\frac{\partial} {\partial x_{i}} m+\frac{1}{\var}w_{i}+f_{i}(\bar{u})-f_{i}(\bar{u}+m),\quad\quad i=1,2,...,d,
\end{aligned}
\end{equation}
and $C_{0}>1$ is a suitably large constant. 

There exists a constant $\eta_{0}>0$ such that if $\|(m_{0},w_{0})\|_{\dot{B}^{\frac{d}{2}-1}_{2,1}\cap\dot{B}^{\frac{d}{2}}_{2,1}}\leq \eta_{0}$, then $(m,w)$ satisfies
  \begin{equation}\label{aprioriestimate}
 \begin{aligned}
 &X(t)\leq C_{0}\|(m_{0}, w_{0})\|_{\dot{B}^{\frac{d}{2}-1}_{2,1}\cap\dot{B}^{\frac{d}{2}}_{2,1}},\quad\quad t\in (0,T).
 \end{aligned}
 \end{equation}
\end{Prop}

Proposition \ref{aprioriJX} will be derived from the linear combination of low-frequency estimates \eqref{lowd2111}-\eqref{lowd2152} and high-frequency estimates  \eqref{Highe}-\eqref{Highe2} established in the next pages.

\subsubsection{Low-frequency analysis}
\leavevmode\par\medbreak

We fix the threshold $J_{\var}$ in \eqref{Jvar} to perform the low-frequency estimates. Inspired by the previous works \cite{c1,c2}, we introduce the effective unknown $z$ defined in \eqref{effective} to rewrite the equations $\eqref{rJXD}_{1}$-$\eqref{rJXD}_{2}$ as 
\begin{equation} \label{JXDZ} 
    \left\{
    \begin{aligned}
    &\frac{\partial}{\partial t}m-\sum_{i=1}^{d}\frac{\partial}{\partial x_{i}}(A_{i} \frac{\partial}{\partial x_{i}} m)=-\sum_{i=1}^{d}\frac{\partial}{\partial x_{i}} z_{i}+\sum_{i=1}^{d}\frac{\partial}{\partial x_{i}} f_{i}(\bar{u}+m),\\
&\frac{\partial}{\partial t}z_{k}+\dfrac{1}{\varepsilon^2}z_{k}= A_{k} \frac{\partial}{\partial x_{k}}\big( \sum_{i=1}^{d}\frac{\partial}{\partial x_{i}}(A_{i}  \frac{\partial}{\partial x_{i}} m) \big)-A_{k} \frac{\partial}{\partial x_{k}} \sum_{i=1}^{d}\frac{\partial}{\partial x_{i}} z_{i}\\
&\quad\quad+ A_{k} \frac{\partial}{\partial x_{k}} (\sum_{i=1}^{d}\frac{\partial}{\partial x_{i}} f_{i}(\bar{u}+m))-\frac{1}{\var}f'_{k}(\bar{u}+m)\sum_{i=1}^{d}\frac{\partial}{\partial x_{i}} w_{i},\quad k=1,2,...,d.
 \end{aligned}
 \right.
\end{equation}
In order to control some nonlinear terms in \eqref{JXDZ}, for example, $\frac{\partial}{\partial x_{i}} f_{i}(\bar{u}+m)$, we need to perform the additional $\dot{B}^{\frac{d}{2}-1}_{2,1}$-estimates in low frequencies compared to the works \cite{c1,c2}.


We first establish the $\dot{B}^{\frac{d}{2}-1}_{2,1}$-estimates of $(m,z)$. For $u$, applying \eqref{lhl} and the regularity estimate \eqref{maximal1} to the parabolic equation $\eqref{JXDZ}_{1}$, we have
\begin{equation}\label{ud21}
\begin{aligned}
&\|m\|_{\widetilde{L}^{\infty}_{t}(\dot{B}^{\frac{d}{2}-1}_{2,1})}^{\ell}+\|m\|_{L^1_{t}(\dot{B}^{\frac{d}{2}+1}_{2,1})}^{\ell}\\
&\lesssim \|m_{0}\|_{\dot{B}^{\frac{d}{2}-1}_{2,1}}^{\ell}+\sum_{i=1}^{d}  \|\frac{\partial}{\partial x_{i}} z_{i}\|_{L^1_{t}(\dot{B}^{\frac{d}{2}-1}_{2,1})}^{\ell}+\sum_{i=1}^{d}\|\frac{\partial}{\partial x_{i}}  f_{i}(\bar{u}+m)\|_{L^1_{t}(\dot{B}^{\frac{d}{2}-1}_{2,1})}^{\ell}
\\&\lesssim \|m_{0}\|_{\dot{B}^{\frac{d}{2}-1}_{2,1}}^{\ell}+2^{J_{\var}}\|z\|_{L^1_{t}(\dot{B}^{\frac{d}{2}-1}_{2,1})}^{\ell}+\| f(\bar{u}+m)-f(\bar{u})\|_{L^1_{t}(\dot{B}^{\frac{d}{2}}_{2,1})}^{\ell}.
\end{aligned}
\end{equation}
As for $z$, it follows from \eqref{lhl} and the regularity estimate \eqref{maximal11} to the equation $\eqref{JXDZ}_{2}$ that
\begin{equation}\label{zd2-1}
\begin{aligned}
&\var\|z\|_{\widetilde{L}^{\infty}_{t}(\dot{B}^{\frac{d}{2}-1}_{2,1})}^{\ell}+\frac{1}{\varepsilon}\|z\|_{L^1_{t}(\dot{B}^{\frac{d}{2}-1}_{2,1})}^{\ell}\\
&\quad\lesssim\|(m_{0},w_{0})\|_{\dot{B}^{\frac{d}{2}-1}_{2,1}}^{\ell}+\|m_{0}\|_{\dot{B}^{\frac{d}{2}}_{2,1}}^{h} + \var 2^{J_{\var}}\| m \|_{L^1_{t}(\dot{B}^{\frac{d}{2}+1}_{2,1})}^{\ell}+ \var 2^{2J_{\var}}\|  z \|_{L^1_{t}(\dot{B}^{\frac{d}{2}-1}_{2,1})}^{\ell}\\
&\quad\quad+\var 2^{J_{\var}}  \| f(\bar{u}+m)-f(\bar{u})\|_{L^1_{t}(\dot{B}^{\frac{d}{2}}_{2,1})}^{\ell} +\sum_{i=1}^{d} \|f'_{k}(\bar{u}+m)\frac{\partial}{\partial x_{i}} w_{i} \|_{L^1_{t}(\dot{B}^{\frac{d}{2}-1}_{2,1})}^{\ell}.
\end{aligned}
\end{equation}
To estimate the nonlinear term $f$, we deduce from \eqref{lhl}, the  quadratic composition estimate \eqref{q1} and the embedding $\dot{B}^{\frac{d}{2}}_{2,1}\hookrightarrow L^{\infty}$ that
\begin{equation}\label{fue}
\left\{
\begin{aligned}
&\|f(\bar{u}+m)-f(\bar{u})\|_{\widetilde{L}^{\infty}_{t}(\dot{B}^{\frac{d}{2}-1}_{2,1})}^{\ell}\\
&\quad\leq C_{u} \|m\|_{L^{\infty}_{t}(\dot{B}^{\frac{d}{2}}_{2,1})}\big(\|m\|_{\widetilde{L}^{\infty}_{t}(\dot{B}^{\frac{d}{2}-1}_{2,1})}^{\ell}+\var\|m\|_{\widetilde{L}^{\infty}_{t}(\dot{B}^{\frac{d}{2}}_{2,1})}^{h}\big)\lesssim X(t)^2,\\
 &   \| f(\bar{u}+m)-f(\bar{u})\|_{L^1_{t}(\dot{B}^{\frac{d}{2}}_{2,1})}^{\ell}\\
 &\quad\leq C_{m} \| m\|_{\widetilde{L}^2_{t}(\dot{B}^{\frac{d}{2}}_{2,1})}^2\lesssim \Big( \big(\| m\|_{\widetilde{L}^2_{t}(\dot{B}^{\frac{d}{2}}_{2,1})}^{\ell} \big)^2+\big(\frac{1}{\var}\| m\|_{\widetilde{L}^2_{t}(\dot{B}^{\frac{d}{2}}_{2,1})}^{h} \big)^2\Big) \lesssim X(t)^2.
\end{aligned}
\right.
\end{equation}
Note that the above constant $C_{m}>0$ in \eqref{fue} depends on $\|m\|_{L^\infty_t(L^\infty)}$ and is uniformly bounded due to \eqref{apriori}. In addition, using \eqref{apriori}, \eqref{F1} and the product law $\dot{B}^{\frac{d}{2}-1}_{2,1}\times\dot{B}^{\frac{d}{2}}_{2,1}\hookrightarrow \dot{B}^{\frac{d}{2}-1}_{2,1}$ ($d\geq2$) in \eqref{uv2}, one sees that the last term on the right-hand side of \eqref{zd2-1} can be controlled by
\begin{equation}\nonumber
\begin{aligned}
&\sum_{i=1}^{d} \|f'_{k}(\bar{u}+m)\frac{\partial}{\partial x_{i}} w_{i} \|_{L^1_{t}(\dot{B}^{\frac{d}{2}-1}_{2,1})}^{\ell}\\
&\quad \lesssim \|f_{k}'(\bar{u}+m)\|_{L^{\infty}_{t}(\dot{B}^{\frac{d}{2}}_{2,1})} \sum_{i=1}^{d}\|\frac{\partial}{\partial x_{i}} w_{i}\|_{L^1_{t}(\dot{B}^{\frac{d}{2}-1}_{2,1})}\\
&\quad \lesssim \|m\|_{\widetilde{L}^{\infty}_{t}(\dot{B}^{\frac{d}{2}}_{2,1})} (\frac{1}{\var} \|w\|_{L^1_{t}(\dot{B}^{\frac{d}{2}}_{2,1})}^{\ell}+\frac{1}{\var^2} \|w\|_{L^1_{t}(\dot{B}^{\frac{d}{2}}_{2,1})}^{h}) \lesssim X(t)^2.
\end{aligned}
\end{equation}
Inserting the above estimates into \eqref{ud21}-\eqref{zd2-1}, we obtain
\begin{equation}\label{lowd2111}
\begin{aligned}
&\|m\|_{\widetilde{L}^{\infty}_{t}(\dot{B}^{\frac{d}{2}-1}_{2,1})}^{\ell}+ \|m\|_{L^1_{t}(\dot{B}^{\frac{d}{2}+1}_{2,1})}^{\ell}+\var\|z\|_{\widetilde{L}^{\infty}_{t}(\dot{B}^{\frac{d}{2}-1}_{2,1})}^{\ell}+\frac{1}{\varepsilon}\|z\|_{L^1_{t}(\dot{B}^{\frac{d}{2}-1}_{2,1})}^{\ell}\\
&\quad\lesssim X_{0}+\var 2^{J_{\var}}\| m \|_{L^1_{t}(\dot{B}^{\frac{d}{2}+1}_{2,1})}^{\ell}+2^{J_{\var}}(1+\var 2^{J_{\var}})\|  z \|_{L^1_{t}(\dot{B}^{\frac{d}{2}-1}_{2,1})}^{\ell}+X(t)^2.
\end{aligned}
\end{equation}
Since the threshold $J_{\var}$ given by \eqref{Jvar} satisfies $\var 2^{J_{\var}}=2^{-k_{0}}$, we choose a suitably large constant $k_{0}>0$ in \eqref{lowd2111} to show
\begin{equation}\label{lowd211}
\begin{aligned}
&\|m\|_{\widetilde{L}^{\infty}_{t}(\dot{B}^{\frac{d}{2}-1}_{2,1})}^{\ell}+ \|m\|_{L^1_{t}(\dot{B}^{\frac{d}{2}+1}_{2,1})}^{\ell}+\var\|z\|_{\widetilde{L}^{\infty}_{t}(\dot{B}^{\frac{d}{2}-1}_{2,1})}^{\ell}+\frac{1}{\varepsilon}\|z\|_{L^1_{t}(\dot{B}^{\frac{d}{2}-1}_{2,1})}^{\ell}\\
&\quad\lesssim X_{0}+X(t)^2.
\end{aligned}
\end{equation}
By interpolation, we also gain
\begin{equation*}
\begin{aligned}
&\|m\|_{\widetilde{L}^{2}_{t}(\dot{B}^{\frac{d}{2}}_{2,1})}^{\ell}\lesssim \big( \|m\|_{\widetilde{L}^{\infty}_{t}(\dot{B}^{\frac{d}{2}-1}_{2,1})}^{\ell}\big)^{\frac{1}{2}}\big( \|m\|_{L^1_{t}(\dot{B}^{\frac{d}{2}+1}_{2,1})}^{\ell}\big)^{\frac{1}{2}}\lesssim  X_{0}+X(t)^2.
\end{aligned}
\end{equation*}
Thence, to recover the information on $v$, we deduce from \eqref{fue}, \eqref{lowd211} and the fact that $ w_{i}= \var z_{i}-\var A_{i}\frac{\partial} {\partial x_{i}} m+\var\big( f_{i}(\bar{u}+m)-f_{i}(\bar{u}) \big)$ ($i=1,2,...,d$) that
\begin{equation}\label{lowd212}
\begin{aligned}
\| w\|_{\widetilde{L}^{\infty}_{t}(\dot{B}^{\frac{d}{2}-1}_{2,1})}^{\ell}
&\lesssim \var \| m\|_{\widetilde{L}^{\infty}_t(\dot{B}^{\frac{d}{2}}_{2,1})}^{\ell}+\var\| z\|_{\widetilde{L}^{\infty}_{t}(\dot{B}^{\frac{d}{2}-1}_{2,1})}^{\ell}+\var\|f(\bar{u}+m)-f(\bar{u})\|_{\widetilde{L}^{\infty}_{t}(\dot{B}^{\frac{d}{2}-1}_{2,1})}^{\ell}\\
&\lesssim \| m\|_{\widetilde{L}^{\infty}_t(\dot{B}^{\frac{d}{2}-1}_{2,1})}^{\ell}+\var\| z\|_{\widetilde{L}^{\infty}_{t}(\dot{B}^{\frac{d}{2}-1}_{2,1})}^{\ell}+\var\|f(\bar{u}+m)-f(\bar{u})\|_{\widetilde{L}^{\infty}_{t}(\dot{B}^{\frac{d}{2}-1}_{2,1})}^{\ell}\\
&\lesssim X_{0}+X(t)^2.
\end{aligned}
\end{equation}
 Similarly, it holds that
 \begin{equation}\label{lowd214}
\begin{aligned}
\frac{1}{\var}\|w\|_{L^{1}_{t}(\dot{B}^{\frac{d}{2}}_{2,1})}^{\ell}
&\lesssim \| m\|_{L^{1}_t(\dot{B}^{\frac{d}{2}+1}_{2,1})}^{\ell}+\|z\|_{L^1_{t}(\dot{B}^{\frac{d}{2}-1}_{2,1})}^{\ell}+\|f(\bar{u}+m)-f(\bar{u})\|_{L^1_{t}(\dot{B}^{\frac{d}{2}}_{2,1})}^{\ell}\\&\lesssim X_{0}+X(t)^2.
\end{aligned}
\end{equation}



Finally, we need to derive the uniform low-frequency estimates of $(m,w)$ in $B^{\frac d2}_{2,1}$ so as to control $f$ in \eqref{fue}. We obtain
\begin{equation}\label{lowd2150}
\begin{aligned}
&\|m\|_{\widetilde{L}^{\infty}_{t}(\dot{B}^{\frac{d}{2}}_{2,1})}^{\ell}+\|m\|_{\widetilde{L}^{2}_{t}(\dot{B}^{\frac{d}{2}+1}_{2,1})}^{\ell}+ \|m\|_{L^1_{t}(\dot{B}^{\frac{d}{2}+2}_{2,1})}^{\ell}+\var\|z\|_{\widetilde{L}^{\infty}_{t}(\dot{B}^{\frac{d}{2}}_{2,1})}^{\ell}+\frac{1}{\varepsilon}\|z\|_{L^1_{t}(\dot{B}^{\frac{d}{2}}_{2,1})}^{\ell}\\
&\quad\lesssim \|(m_{0},w_{0})\|_{\dot{B}^{\frac {d}{2}}_{2,1}}+\|  f(\bar{u}+m)-f(\bar{u})\|^\ell_{L^1_{t}(\dot{B}^{\frac d2+1}_{2,1})}+\sum_{i=1}^{d} \|f'_{k}(\bar{u}+m)\frac{\partial}{\partial x_{i}} w_{i} \|_{L^1_{t}(\dot{B}^{\frac{d}{2}}_{2,1})}^{\ell}.
\end{aligned}
\end{equation}
By \eqref{apriori} and \eqref{q1} one has that
\begin{equation}\nonumber
\begin{aligned}
&\|  f(\bar{u}+m)-f(\bar{u})\|^\ell_{L^1_{t}(\dot{B}^{\frac d2+1}_{2,1})}\\
&\quad \lesssim \|m\| _{L^{\infty}_{t}(L^{\infty})}( \|m\|_{L^1_{t}(\dot{B}^{\frac{d}{2}+1}_{2,1})}^{\ell}+\frac{1}{\var}\|m\|_{L^1_{t}(\dot{B}^{\frac{d}{2}}_{2,1})}^{h})\\
&\quad \lesssim \|m\| _{L^{\infty}_{t}(\dot{B}^{\frac{d}{2}}_{2,1})}( \|m\|_{L^1_{t}(\dot{B}^{\frac{d}{2}+1}_{2,1})}^{\ell}+\frac{1}{\var^2}\|m\|_{L^1_{t}(\dot{B}^{\frac{d}{2}}_{2,1})}^{h})\lesssim X(t)^2.
\end{aligned}
\end{equation}
And it also holds that
\begin{equation}\nonumber
\begin{aligned}
&\sum_{i=1}^{d} \|f'_{k}(\bar{u}+m)\frac{\partial}{\partial x_{i}} w_{i} \|_{L^1_{t}(\dot{B}^{\frac{d}{2}}_{2,1})}^{\ell}\\
&\quad \lesssim \frac{1}{\var}\sum_{i=1}^{d} \|f'_{k}(\bar{u}+m)\frac{\partial}{\partial x_{i}} w_{i} \|_{L^1_{t}(\dot{B}^{\frac{d}{2}-1}_{2,1})}^{\ell}\\
&\quad \lesssim \|m\|_{\widetilde{L}^{\infty}_{t}(\dot{B}^{\frac{d}{2}}_{2,1})} (\frac{1}{\var} \|w\|_{L^1_{t}(\dot{B}^{\frac{d}{2}}_{2,1})}^{\ell}+\frac{1}{\var^2} \|w\|_{L^1_{t}(\dot{B}^{\frac{d}{2}}_{2,1})}^{h}) \lesssim X(t)^2.
\end{aligned}
\end{equation}
Gathering the above estimates, we gain
\begin{equation}\label{lowd2151}
\begin{aligned}
&\|m\|_{\widetilde{L}^{\infty}_{t}(\dot{B}^{\frac{d}{2}}_{2,1})}^{\ell}+\|m\|_{\widetilde{L}^{2}_{t}(\dot{B}^{\frac{d}{2}+1}_{2,1})}^{\ell}+ \|m\|_{L^1_{t}(\dot{B}^{\frac{d}{2}+2}_{2,1})}^{\ell}+\var\|z\|_{\widetilde{L}^{\infty}_{t}(\dot{B}^{\frac{d}{2}}_{2,1})}^{\ell}+\frac{1}{\varepsilon}\|z\|_{L^1_{t}(\dot{B}^{\frac{d}{2}}_{2,1})}^{\ell}\\
&\quad\lesssim X_{0}+X(t)^2.
\end{aligned}
\end{equation}
Similarly to \eqref{lowd212}-\eqref{lowd214}, with the help of \eqref{lowd2151} and $\eqref{JXDZ}_{2}$, one can show
\begin{equation}\label{lowd2152}
\begin{aligned}
&\| w\|_{\widetilde{L}^{\infty}_{t}(\dot{B}^{\frac{d}{2}}_{2,1})}^{\ell}+\frac{1}{\var}\|w\|_{L^1_{t}(\dot{B}^{\frac{d}{2}+1}_{2,1})}^{\ell}\lesssim X_{0}+X(t)^2.
\end{aligned}
\end{equation}

\subsubsection{High-frequency analysis} \label{subsectionhigh}
\leavevmode\par\medbreak

The dissipative structures of \eqref{JXD} in high frequencies are analysed as follows. For $j\geq J_\varepsilon-1$, applying $\ddj$ to \eqref{JXD}, we obtain
\begin{equation} \label{JXD1}
    \left\{
    \begin{aligned}
&\frac{\partial} {\partial t}\ddj m+\frac{1}{\var}\sum_{i=1}^{d} \frac{\partial} {\partial x_{i}} \ddj w_{i}=0, \\
&\frac{\partial} {\partial t} \ddj w_{i}+\frac{1}{\var} A_{i}\frac{\partial} {\partial x_{i}} \ddj m+\frac{1}{\var^2}\ddj w_{i}= \frac{1}{\var} \ddj \big(f_{i}(\bar{u}+m)-f_{i}(\bar{u}) \big),\quad  i=1,2,...,d.
\end{aligned}
\right.
\end{equation}
Taking the $L^2$-inner product of $\eqref{JXD1}_{1}$ with $\ddj u $ and integrating by parts, we have
\begin{equation}\label{Luj}
\frac{1}{2}\dfrac{d}{dt}\|\ddj m\|_{L^2}^2-\frac{1}{\var}\sum_{i=1}^{d} \int_{\mathbb{R}^{d}} \ddj w_{i} \cdot\frac{\partial} {\partial x_{i}} \ddj m dx=0.
\end{equation}
Meanwhile, taking the $L^2$-inner product of $\eqref{JXD1}_{1}$ with $\dfrac{1}{a_{i}} \ddj v_{i} $ and summing over $1\leq i\leq d$, we also get
\begin{equation}\label{Luj1}
\begin{aligned}
  &\sum_{i=1}^{d}\frac{1}{2a_{i}}\dfrac{d}{dt} \|\ddj  w_{i}\|_{L^2}^2+\frac{1}{\var}\sum_{i=1}^{d} \int_{\mathbb{R}^{d}} \ddj w_{i} \cdot \frac{\partial}{\partial x_{i}} \ddj m dx+\frac{1}{\var^2}\sum_{i=1}^{d}\frac{1}{a_{i}} \| \ddj w_{i}\|_{L^2}^2\\
  &\quad=\frac{1}{\var}\sum_{i=1}^{d}\frac{1}{a_{i}}\int_{\mathbb{R}^{d}} \ddj \big(f_{i}(\bar{u}+m)-f_{i}(\bar{u}) \big) \cdot \ddj w_{i}dx.
  \end{aligned}
\end{equation}
In addition, in order to derive dissipative estimates for $u$, one has to perform the following cross estimates
\begin{equation}\label{Lvj}
\begin{aligned}
&\frac{1}{\var}\sum_{i=1}^{d}\frac{1}{a_{i}}\dfrac{d}{dt}\int_{\mathbb{R}^{d}} \ddj  w_{i} \cdot \frac{\partial} {\partial x_{i}} \ddj m_{i} dx+\frac{1}{\varepsilon^2}\sum_{i=1}^{d}\|\frac{\partial} {\partial x_{i}} \ddj  m_{i}\|_{L^2}^2\\
&\quad\quad-\frac{1}{\varepsilon^2} \sum_{i=1}^{d} \frac{1}{a_{i}} \int_{\mathbb{R}^{d}} \frac{\partial} {\partial x_{i}}  \ddj w_i  \sum_{i=1}^{k} \frac{\partial} {\partial x_{k}}  \ddj w_{k}  dx+\frac{1}{\varepsilon^3}\sum_{i=1}^{d} \frac{1}{a_{i}}\int_{\mathbb{R}^{d}} \ddj w_{i}\cdot \frac{\partial} {\partial x_{i}} \ddj m_{i} dx\\
&\quad=\frac{1}{\varepsilon}\sum_{i=1}^{d} \frac{1}{a_{i}} \int_{\mathbb{R}^{d}} \ddj \big(f_{i}(\bar{u}+m)-f_{i}(\bar{u})\big) \cdot  \frac{\partial} {\partial x_{i}} \ddj m_{i}  dx.
\end{aligned}
\end{equation}
For any $j\geq J_{\var}-1$, the linear combination of \eqref{Luj}-\eqref{Lvj} leads to
\begin{equation}
\begin{aligned}
    &\frac{d}{dt}\mathcal{L}_{j}^2(t)+\frac{1}{\varepsilon^2}\mathcal{H}^2_{j}(t)\\
    &\leq \sum_{i=1}^{d}\frac{1}{a_{i}}\|\ddj \big(f_{i}(\bar{u}+m)-f_{i}(\bar{u}) \big) \|_{L^2} \big( \frac{1}{\var}\|\ddj w\|_{L^2}+ \frac{2^{-2j}\eta}{\varepsilon^2}  \|\frac{\partial} {\partial x_{i}} \ddj m_{i}\|_{L^2}\big),\label{Ly}
    \end{aligned}
    \end{equation}
where the Lyapunov functional $\mathcal{L}^2_j(t)$ and its dissipation rate $\mathcal{H}^2_j(t)$ are defined by
\begin{equation}\nonumber
\left\{
\begin{aligned}
&\mathcal{L}^2_j(t):=\frac{1}{2}\|\ddj m \|_{L^2}^2+\sum_{i=1}^{d}\frac{1}{2a_{i}}\|\ddj w_{i}\|_{L^2}^2+\frac{2^{-2j}\eta}{\var}\sum_{i=1}^{d}\frac{1}{a_{i}}\int_{\mathbb{R}^{d}} \ddj  w_{i} \cdot \frac{\partial} {\partial x_{i}} \ddj m_i dx ,\\
&\mathcal{H}^2_{j}(t):=\frac{1}{\var^2}\sum_{i=1}^{d}\frac{1}{a_{i}} \| \ddj w_{i}\|_{L^2}^2+\frac{2^{-2j}\eta}{\varepsilon^2} \sum_{i=1}^{d} \Big( \|\frac{\partial} {\partial x_{i}} \ddj  m_i\|_{L^2}^2\\
&\quad\quad\quad\quad- \frac{1}{a_{i}} \int_{\mathbb{R}^{d}} \frac{\partial} {\partial x_{i}}  \ddj w_{i}  \cdot \sum_{i=1}^{k} \frac{\partial} {\partial x_{k}}  \ddj w_{k}  dx+ \frac{1}{\var a_{i}}\int_{\mathbb{R}^{d}} \ddj w_{i}\cdot \frac{\partial} {\partial x_{i}} \ddj m_i dx \Big),
\end{aligned}
\right.
\end{equation}
for a constant $\eta>0$ to be adjusted later on.

\begin{Lemma} \label{Lyaeq}
For any $j\geq J_{\var}-1$, there exists a suitable small constant $\eta>0$ such that
\begin{equation}\label{LHsim}
\left\{
\begin{aligned}
&\mathcal{L}^2_j(t)\sim \|\ddj(m,w)\|_{L^2}^2,\\
& \mathcal{H}^2_{j}(t)\gtrsim \frac{1}{\var^2} \|\ddj(m,w)\|_{L^2}^2.
\end{aligned}
\right.
\end{equation}
\end{Lemma}

\begin{proof}
By Bernstein’s inequality and $2^{-j}\lesssim \var$ for $j\geq J_{\var}-1$, it follows that
\begin{align*} 
   \frac{2^{-2j}}{\var}\int_{\mathbb{R}^{d}} \ddj  w_{i} \cdot \frac{\partial} {\partial x_{i}} \ddj m_i dx  &\lesssim \frac{2^{-j}}{\var}\|\ddj (m, w)\|_{L^2}^2\lesssim \|\ddj (m, w)\|_{L^2}^2,
\end{align*}
which implies
\begin{equation}\nonumber
\begin{aligned}
&\frac{1}{2}(1-C\eta)  \|\ddj(m,w)\|_{L^2}^2\leq \mathcal{L}^2_j(t)\leq \frac{1}{2}(1+C\eta)\|\ddj(m,w)\|_{L^2}^2
\end{aligned}
\end{equation}
for some constant $C>1$ independent of time and $\var$.

On the other hand, for any $j\geq J_{\var}-1$, since the cross term satisfies
\begin{equation}\nonumber
\begin{aligned}
&\frac{2^{-2j}}{\varepsilon^2} \sum_{i=1}^{d} \Big( \|\frac{\partial} {\partial x_{i}} \ddj  m_i\|_{L^2}^2- \frac{1}{a_{i}} \int_{\mathbb{R}^{d}} \frac{\partial} {\partial x_{i}}  \ddj w_{i} \cdot \sum_{i=1}^{k} \frac{\partial} {\partial x_{k}}  \ddj w_{k}  dx\\
&\quad\quad+ \frac{1}{\var a_{i}}\int_{\mathbb{R}^{d}} \ddj w_{i}\cdot \frac{\partial} {\partial x_{i}} \ddj m_i dx \Big)\\
&\quad\gtrsim  \frac{1}{\varepsilon^2} \|\ddj m\|_{L^2}^2- \frac{1}{\varepsilon^2}\|\ddj w\|_{L^2},
\end{aligned}
\end{equation}
the following estimate holds:
\begin{equation}\nonumber
\begin{aligned}
&\mathcal{H}^2_{j}(t)\geq   \frac{1}{\var^2}(\frac{1}{\max_{1\leq i\leq d}a_{i}}-C\eta)\|\ddj w\|_{L^2}^2 +\frac{\eta}{C\var^2} \|\ddj m\|_{L^2}^2.
\end{aligned}
\end{equation}
By choosing $\eta=\dfrac{1}{ 2C\max_{1\leq i\leq d}a_{i} }$, we get \eqref{LHsim}.
\end{proof}

\vspace{5ex}

Then, by virtue of \eqref{Ly}, the Bernstein inequality and Lemma \ref{Lyaeq}, we end up with
\begin{align}\label{LyaHf}
    \frac{d}{dt}\mathcal{L}_{j}^2(t)+\frac{1}{\var^2}\mathcal{L}_{j}^2(t)&\lesssim  \frac{1}{\varepsilon }\|\ddj \big(f_{i}(\bar{u}+m)-f_{i}(\bar{u}) \big) \|_{L^2}  \mathcal{L}_{j}(t),\quad j\geq J_{\var}-1.
\end{align}
Dividing the two sides of \eqref{LyaHf} by $(\mathcal{L}_{j}^2(t)+\eta)^{\frac{1}{2}}$, integrating the resulting equation over $[0,t]$ and then taking the limit as $\eta\rightarrow0$, we have
\begin{align*}
   &\|\ddj( m,w)\|_{L^2}+\frac{1}{\varepsilon^2}\int_0^t\|\ddj( m,w)\|_{L^2}d\tau\\
   &\quad\leq  \|\dot{\Delta}_{j}( m_{0}, w_{0})\|_{L^2}+ \frac{1}{\varepsilon }\int_0^t\|\ddj \big(f_{i}(\bar{u}+m)-f_{i}(\bar{u}) \big) \|_{L^2}d\tau,\quad j\geq J_{\var}-1.
\end{align*}
Then, a classical procedure leads to
\begin{equation}\nonumber
\begin{aligned}
&\|( m,w)\|_{\widetilde{L}^{\infty}_{t}(\dot{B}^{\frac{d}{2}}_{2,1})}^{h}+\frac{1}{\varepsilon^2}\|(m,w)\|_{L^1_{t}(\dot{B}^{\frac{d}{2}}_{2,1})}^{h}\lesssim \|( m_0,w_0)\|_{\dot{B}^{\frac{d}{2}}}^{h}+\frac{1}{\varepsilon }\|f(\bar{u}+m)-f(\bar{u})\|_{L^1_t(B^{\frac d2}_{2,1})}^h.
\end{aligned}
\end{equation}
Since $f(u)$ is quadratic, an application of \eqref{apriori} and \eqref{q2} with $\sigma=\frac{d}{2}+1$ yields
\begin{equation}\nonumber
\begin{aligned}
\|f(\bar{u}+m)-f(\bar{u})\|_{\dot{B}^{\frac{d}{2}}_{2,1}}^{h}&\lesssim \|m\|_{L^{\infty}}\big{(} \varepsilon\|m\|_{\dot{B}^{\frac{d}{2}+1}_{2,1}}^{\ell}+\|m\|_{\dot{B}^{\frac d2}_{2,1}}^{h}\big{)},
\end{aligned}
\end{equation}
which gives rise to
\begin{equation}\nonumber
\begin{aligned}
&\frac{1}{\varepsilon}\|f(\bar{u}+m)-f(\bar{u})\|_{L^1_t(B^{\frac d2}_{2,1})}^h\\
&\quad\lesssim\|m\|_{\widetilde{L}^{2}_{t}(\dot{B}^{\frac{d}{2}}_{2,1})}\big{(} \|m\|_{\widetilde{L}^2_{t}(\dot{B}^{\frac{d}{2}+1}_{2,1})}^{\ell}+\frac{1}{\var}\|m\|_{\widetilde{L}^2_{t}(\dot{B}^{\frac{d}{2}}_{2,1})}^{h}\big)\\
&\quad\lesssim \big{(} \|m\|_{\widetilde{L}^2_{t}(\dot{B}^{\frac{d}{2}+1}_{2,1})}^{\ell}+\frac{1}{\var}\|m\|_{\widetilde{L}^2_{t}(\dot{B}^{\frac{d}{2}}_{2,1})}^{h}\big)^2\lesssim X(t)^2.
\end{aligned}
\end{equation}
Therefore, it holds that
\begin{equation}\label{Highe}
\begin{aligned}
\|( m,w)\|_{\widetilde{L}^{\infty}_{t}(\dot{B}^{\frac{d}{2}}_{2,1})}^{h}+\frac{1}{\varepsilon^2}\|(m,w)\|_{L^1_{t}(\dot{B}^{\frac{d}{2}}_{2,1})}^{h}&\lesssim \|(m_{0},w_{0})\|_{\dot{B}^{\frac{d}{2}}_{2,1}}^{h}+X(t)^2.
\end{aligned}
\end{equation}
The above estimates \eqref{Highe} and \eqref{lhl} also give rise to
\begin{equation}\label{Highe1}
\begin{aligned}
&\|( m,w)\|_{\widetilde{L}^{\infty}_{t}(\dot{B}^{\frac{d}{2}-1}_{2,1})}^{h}\lesssim \var \|( m,w)\|_{\widetilde{L}^{\infty}_{t}(\dot{B}^{\frac{d}{2}}_{2,1})}^{h}\lesssim  \|(m_{0},w_{0})\|_{\dot{B}^{\frac{d}{2}}_{2,1}}^{h}+X(t)^2.
\end{aligned}
\end{equation}
By H${\rm\ddot{o}}$lder's inequality, one has
\begin{equation}\nonumber
\begin{aligned}
& \|m\|_{\widetilde{L}^{2}_{t}(\dot{B}^{\frac{d}{2}}_{2,1})}^{h}\lesssim \big(\|m\|_{\widetilde{L}^{\infty}_{t}(\dot{B}^{\frac{d}{2}}_{2,1})}^{h} \big) \big( \frac{1}{\var^2}\|m\|_{L^{1}_{t}(\dot{B}^{\frac{d}{2}}_{2,1})}^{h}\big)^{\frac{1}{2}}\lesssim \|(m_{0},w_{0})\|_{\dot{B}^{\frac{d}{2}}_{2,1}}^{h}+X(t)^2.
\end{aligned}
\end{equation}

In order to justify the relaxation limit, we need to show that the low-frequency damped mode $z$ defined by \eqref{effective} also satisfied uniform bounds in high frequencies. Indeed, from \eqref{Highe1} and \eqref{lhl}, we have
\begin{equation}\label{Highe2}
\begin{aligned}
    \frac{1}{\varepsilon^2}\|z\|_{L^1_t(B^{\frac d2-1}_{2,1})}^h&\lesssim\frac{1}{\varepsilon^2}\| m\|_{L^1_t(\dot{B}^{\frac d2}_{2,1})}^{h}+\frac{1}{\varepsilon^3}\|w\|_{L^1_t(\dot{B}^{\frac d2-1}_{2,1})}^{h}+\frac{1}{\varepsilon^2}\|f(\bar{u}+m)-f(\bar{u})\|_{L^1_t(\dot{B}^{\frac d2-1}_{2,1})}^{h}
    \\&\lesssim \frac{1}{\varepsilon^2}\| m\|_{L^1_t(\dot{B}^{\frac d2}_{2,1})}^{h}+\frac{1}{\var^2}\|w\|_{L^1_t(\dot{B}^{\frac d2}_{2,1})}^{h}+\frac{1}{\var}\|f(\bar{u}+m)-f(\bar{u})\|_{L^1_t(\dot{B}^{\frac d2}_{2,1})}^{h}\\
    &\lesssim X_{0}+X(t)^2.
\end{aligned}
\end{equation}

\subsection{Proof of global well-posedness
}\leavevmode\par\medbreak

In this section we conclude the proof of Theorem \ref{Thm1}. For any $q\geq1$, we consider the approximate scheme for $q\geq1$ as follows
\begin{equation}\label{app}
\begin{aligned}
\frac{d}{dt}\left(                 
  \begin{matrix}  
    m^{q} \\
    w_{1}^{q}\\
    ...\\
    w_{d}^{q}
  \end{matrix}
\right)=\left(
  \begin{matrix} 
\frac{1}{\var} \sum_{i=1}^{d}\frac{\partial}{\partial x_{i}} w_{i}^{q}\\
-\frac{1}{\var}A_{1}\frac{\partial}{\partial x_{1}} m^{q}-\frac{1}{\var^2} w_{1}^{q}+\frac{1}{\var}\dot{\mathbb{E}}_{q} ( F_{1}(\bar{u}+m^{q})-F_{1}(\bar{u}))\\
...\\
-\frac{1}{\var}A_{d}\frac{\partial}{\partial x_{d}} m^{q}-\frac{1}{\var^2} w_{m}^{q}+\frac{1}{\var^2}\dot{\mathbb{E}}_{q} ( F_{d}(\bar{u}+m^{q})-F_{d}(\bar{u}))\\
    \end{matrix}
\right),
 \end{aligned}
\end{equation}
with the initial data
\begin{equation}\label{appd}
\begin{aligned}
   ( m^{q},w^{q})|_{t=0}= \dot{\mathbb{E}}_{q} (m_{0}, w_{0}),
 \end{aligned}
\end{equation}
where $\dot{\mathbb{E}}_{q}: L^2\rightarrow L^2_{q}$ is is the Friedrichs projector with  $L^2_{q}$ the set of $L^2$ functions spectrally supported in the annulus $\{\frac{1}{q}\leq|\xi|\leq q\}$. Since (\ref{app}) is a system of ordinary differential equations in $L^2_{q}\times L^2_{q}$ and locally Lipschitz with respect to the variable $(m^{q},w^{q})$ for every $n\geq 1$ due to the Bernstein inequality, it follows by the Cauchy-Lipschitz theorem in \cite[Page 124]{bahouri1} that there exists a time $T^{*}_{n}>0$ such that the problem (\ref{app})-\eqref{appd} for $q\geq1$ admits a unique solution $(m^{q},w^{q})\in C([0,T^{*}_{n}];L^2_{q})$ on $[0,T_{n}^{*}]$.

By virtue of Proposition \ref{aprioriJX} and the standard continuity arguments, we can extend the solution $(m^{q},w^{q})$ globally in time and prove that $(m^{q},w^{q})$ satisfies the uniform estimates \eqref{aprioriestimate} for any $t>0$ and $q\geq1$.


Then, we prove the strong convergence of the approximate sequence $(u^{q},v^{q})$. Let $(\widetilde{m}^{q},\widetilde{w}^{q}):=(m^{q+1}-m^{q}, w^{q+1}-w^{q})$ which satisfies
\begin{equation} \nonumber
\left\{
    \begin{aligned}
&\frac{\partial} {\partial t}\widetilde{m}^{q}+\frac{1}{\var}\sum_{i=1}^{d} \frac{\partial} {\partial x_{i}} \widetilde{w}^{q}_{i}=0, \\
&\frac{\partial} {\partial t} \widetilde{w}^{q}_{i}+\frac{1}{\var} A_{i}\frac{\partial} {\partial x_{i}} \widetilde{m}^{q}+\frac{1}{\var^2}\widetilde{w}^{q}_{i}=\frac{1}{\var}\big( f_{i}(\bar{u}+m^{q+1})-f_{i}(\bar{u}+m^{q}) \big),\quad i=1,2,...,d.
    \end{aligned}
    \right.
\end{equation}
Therefore, one can obtain
\begin{equation}\nonumber
    \begin{aligned}
    &\frac{1}{2} \frac{d}{dt} \big( \|\ddj \widetilde{m}^{q} \|_{L^2}^2+\sum_{i=1}^{d}\frac{1}{a_{i}}\|\ddj \widetilde{w}_{i}^{q}\|_{L^2}^2 \big)+\frac{1}{\var^2}\sum_{i=1}^{d}\frac{1}{ 2 a_{i}} \| \ddj \widetilde{w}_{i}^{q}\|_{L^2}^2\\
    &~\leq \frac{\max_{1\leq i\leq d} a_{i} }{2 }\|\dot{\Delta}_{j}\big(f(\bar{u}+m^{q+1})-f(\bar{u}+m^{q})\big)\|_{L^2}^2,
    \end{aligned}
\end{equation}
which yields
\begin{equation}\label{uniquee}
    \begin{aligned}
    &\|(\widetilde{m}^{q},\widetilde{w}^{q})\|_{\widetilde{L}^{\infty}_{t}(\dot{B}^{\frac{d}{2}}_{2,1})}^2+\frac{1}{\var^2}\|\widetilde{w}^{q}\|_{\widetilde{L}^2_{t}(\dot{B}^{\frac{d}{2}}_{2,1})}^2\\
    &\quad\lesssim \|(\dot{\mathbb{E}}_{n+1}-\dot{\mathbb{E}}_{q})( m_{0},w_{0})\|_{\dot{B}^{\frac{d}{2}}_{2,1}}^2+\int_{0}^{t}\|f(\bar{u}+m^{q+1})-f(\bar{u}+m^{q})\|_{\dot{B}^{\frac{d}{2}}_{2,1}}^2d\tau.
        \end{aligned}
\end{equation}
Thanks to the composition estimate \eqref{F2}, we get
\begin{equation}\label{uniquee1}
    \begin{aligned}
    \|f(\bar{u}+m^{q+1})-f(\bar{u}+m^{q})\|_{\dot{B}^{\frac{d}{2}}_{2,1}}\lesssim \|\widetilde{m}^{q}\|_{\dot{B}^{\frac{d}{2}}_{2,1}}.
        \end{aligned}
\end{equation}
By \eqref{uniquee}-\eqref{uniquee1} and the Gr\"onwall inequality, there exists a limit $(m,w)$ such that as $n\rightarrow\infty$, the following convergence holds:
\begin{equation}
    \begin{aligned}
    &(m^{q}, w^{q})\rightarrow (m,w)\quad\text{strongly in}\quad L^{\infty}(0,T;\dot{B}^{\frac{d}{2}}_{2,1}),\quad \forall T>0. 
    \end{aligned}
\end{equation}
Thus, it is easy to prove that the limit $(m,w)$ solves \eqref{rJXD} in the sense of distributions, and thanks to the uniform estimates $\eqref{aprioriestimate}$ and the Fatou property,  $(u,v):=(\bar{u}+m,\bar{v}+\frac{1}{\var}w)$ is indeed a global strong solution to the Cauchy problem of System \eqref{JXD} subject to the initial data $(u_{0},v_{0})$ and satisfies the estimates (\ref{Xt}). Similarly to the convergence of the approximate sequence, we can show the uniqueness in the space $L^{\infty}(0,T;\dot{B}^{\frac{d}{2}}_{2,1})$. The details are omitted here.

\section{Proof of Theorem \ref{Thm3}}\label{section4}
\subsection{Global existence of the viscous conservation laws}

Here we establish the global a-priori estimates of any solution $u^{*}$ to System \eqref{uheat}. By virtue of the regularity estimates \eqref{maximal1} applied to \eqref{uheat} and 
the composition estimate \eqref{F1}, we have
\begin{equation}\label{321}
\begin{aligned}
&\|u^*-\bar{u}\|_{\widetilde{L}^{\infty}_{t}(\dot{B}^{\frac{d}{2}-1}_{2,1}\cap\dot{B}^{\frac{d}{2}}_{2,1})}+\|u^*-\bar{u}\|_{L^1_{t}(\dot{B}^{\frac{d}{2}+1}_{2,1}\cap\dot{B}^{\frac{d}{2}+2}_{2,1})}\\
&\quad\lesssim  \|u_{0}^*-\bar{u}\|_{\dot{B}^{\frac{d}{2}-1}_{2,1}\cap\dot{B}^{\frac{d}{2}}_{2,1}}+\|f(u^{*})-f(\bar{u})\|_{L^1_{t}(\dot{B}^{\frac{d}{2}}_{2,1}\cap\dot{B}^{\frac{d}{2}+1}_{2,1})}.
\end{aligned}
\end{equation}
Together with an interpolation inequality it yields
\begin{equation}\label{3211}
\begin{aligned}
&\|u^*-\bar{u}\|_{\widetilde{L}^2_{t}(\dot{B}^{\frac{d}{2}}_{2,1}\cap\dot{B}^{\frac{d}{2}+1}_{2,1})}\lesssim \|u_{0}^*-\bar{u}\|_{\dot{B}^{\frac{d}{2}-1}_{2,1}\cap\dot{B}^{\frac{d}{2}}_{2,1}}+\|f(u^{*})-f(\bar{u})\|_{L^1_{t}(\dot{B}^{\frac{d}{2}}_{2,1}\cap\dot{B}^{\frac{d}{2}+1}_{2,1})}.
\end{aligned}
\end{equation}
Using the composition estimates \eqref{q1}-\eqref{q2} and interpolation \eqref{inter}, there holds
\begin{equation}\nonumber
\left\{
    \begin{aligned}
    &\|f(u^{*})-f(\bar{u})\|_{L^1_{t}(\dot{B}^{\frac{d}{2}}_{2,1})}\lesssim \|u^*-\bar{u}\|_{\widetilde{L}^2_{t}(\dot{B}^{\frac{d}{2}}_{2,1})}^2 ,\\
    &\|f(u^{*})-f(\bar{u})\|_{L^1_{t}(\dot{B}^{\frac{d}{2}+1}_{2,1})}\lesssim \|u^*-\bar{u}\|_{\widetilde{L}^2_{t}(\dot{B}^{\frac{d}{2}}_{2,1})}\|u^*-\bar{u}\|_{\widetilde{L}^2_{t}(\dot{B}^{\frac{d}{2}+1}_{2,1})}.
    \end{aligned}
    \right.
\end{equation}

Thus the left-hand sides of \eqref{321}-\eqref{3211} can be bounded
for all $t>0$ in terms of initial data verifying \eqref{smalllimit}. Thence by the Friedrichs approximation and convergence arguments, one can prove the global existence of a solution $u^{*}$ to the Cauchy problem of System \eqref{uheat} associated to the initial data $u^*_0$ and the proof of uniqueness follows from similar estimates. In addition, the law \eqref{Darcy} gives rise to
\begin{equation}\nonumber
\begin{aligned}
&\|v^*-\bar{v}\|_{L^1_{t}(\dot{B}^{\frac{d}{2}}_{2,1})}\lesssim \|u^{*}-\bar{u}\|_{L^1_{t}(\dot{B}^{\frac{d}{2}+1}_{2,1})}+\|f(u^{*})-f(\bar{u})\|_{L^1_{t}(\dot{B}^{\frac{d}{2}}_{2,1})}\lesssim  \|u_{0}^*-\bar{u}\|_{\dot{B}^{\frac{d}{2}-1}_{2,1}\cap\dot{B}^{\frac{d}{2}}_{2,1}}.
\end{aligned}
\end{equation}

\subsection{Convergence rate}

We are ready to estimate the difference between $u^*$ the solution of \eqref{uheat} and the solution $u$ of \eqref{JXD} when $u_{0}$ and $u_0^*$ satisfies \eqref{error}. Defining $(\delta u, \delta v):=(u-u^*,v-v^*)$ and using $\eqref{JXDZ}$, we have
\begin{equation} \label{diff1} 
\left\{
\begin{aligned}
&\d_t\delta u-\sum_{i=1}^{d}(A_{i}\frac{\partial}{\partial x_{i}} \delta u)=-\sum_{i=1}^{d}\frac{\partial}{\partial x_{i}} z_{i}+\sum_{i=1}^{d}\frac{\partial}{\partial x_{i}}\big(f_{i}(u)-f_{i}(u^*) \big),\\
&\delta v_{i}=-A_{i}\dfrac{\partial} {\partial x_{i}}  \delta u+f_{i}(u)-f_{i}(u^*)+Z_i,\quad i=1,2,...,d,
\end{aligned}
\right.
\end{equation}
with $Z_i=A_{i}\dfrac{\partial} {\partial x_{i}} u+v_{i}-f_{i}(u)$. Applying  \eqref{Xt} and \eqref{maximal1}  to the equation $\eqref{diff1}_{1}$, $\delta u$ can be estimated as follows:
\begin{equation}\label{deltae1}
\begin{aligned}
&\|\delta u\|_{\widetilde{L}^{\infty}_{t}(\dot{B}^{\frac{d}{2}-2}_{2,1})}+\|\delta u\|_{L^1_{t}(\dot{B}^{\frac{d}{2}}_{2,1})}\\
&\quad\lesssim \|u_{0}-u_{0}^{*}\|_{\dot{B}^{\frac{d}{2}-2}_{2,1}}+\sum_{i=1}^{d}\|\frac{\partial}{\partial x_{i}} Z_{i}\|_{L^1_{t}(\dot{B}^{\frac{d}{2}-2}_{2,1})}+\sum_{i=1}^{d}\|\frac{\partial}{\partial x_{i}}\big( f_{i}(u)-f_{i}(u^*)\big)\|_{L^1_{t}(\dot{B}^{\frac{d}{2}-2}_{2,1})}
\\ &
\quad\lesssim  \var+\sum_{i=1}^{d}\| Z_{i}\|_{L^1_{t}(\dot{B}^{\frac{d}{2}-1}_{2,1})}+\sum_{i=1}^{d}\| f_{i}(u)-f_{i}(u^*)\|_{L^1_{t}(\dot{B}^{\frac{d}{2}-1}_{2,1})}
\\ &\quad\lesssim \varepsilon+\sum_{i=1}^{d}\| f_{i}(u)-f_{i}(u^*)\|_{L^1_{t}(\dot{B}^{\frac{d}{2}-1}_{2,1})}.
\end{aligned}
\end{equation}
As $f_{i}$ is at least quadratic due to \eqref{Assumptionf}, there exists a function $g_{i}$ such that $f_{i}(u)-f_{i}(\bar{u})=(u-\bar{u})g_{i}(u)$ and $g_{i}(\bar{u})=0$.
Thus, we have
$$ f_{i}(u)-f_{i}(u^{*})=\delta u g_{i}(u)+(u^*-\bar{u})\big( g_{i}(u)-g_{i}(u^{*})\big).$$ It follows from \eqref{Xt}, \eqref{Xtlimit} and \eqref{F1}-\eqref{F2} that
\begin{equation}\label{deltae11}
\begin{aligned}
   & \| f_{i}(u)-f_{i}(u^{*})\|_{L^1_{t}(\dot{B}^{\frac{d}{2}-1}_{2,1})} \\
   &\quad \leq \| \delta u g_{i}(u)\|_{L^1_{t}(\dot{B}^{\frac{d}{2}-1}_{2,1})}+ \| (u^*-\bar{u})\big( g_{i}(u)-g_{i}(u^{*})\big)\|_{L^1_{t}(\dot{B}^{\frac{d}{2}-1}_{2,1})}
    \\&\quad \lesssim \| \delta u \|_{\widetilde{L}^2_{t}(\dot{B}^{\frac{d}{2}-1}_{2,1})}\|g_{i}(u)\|_{\widetilde{L}^2_{t}(\dot{B}^{\frac{d}{2}}_{2,1})}+\| u^*-\bar{u} \|_{\widetilde{L}^\infty_{t}(\dot{B}^{\frac{d}{2}}_{2,1})}\| g(u)-g(u^{*})\|_{L^1_{t}(\dot{B}^{\frac{d}{2}-1}_{2,1})}
    \\&\quad \lesssim\| \delta u \|_{\widetilde{L}^2_{t}(\dot{B}^{\frac{d}{2}-1}_{2,1})}\|u-\bar{u}\|_{\widetilde{L}^2_{t}(\dot{B}^{\frac{d}{2}}_{2,1})}+\| u^* -\bar{u}\|_{\widetilde{L}^\infty_{t}(\dot{B}^{\frac{d}{2}}_{2,1})}\|\delta u\|_{L^1_{t}(\dot{B}^{\frac{d}{2}-1}_{2,1})}\\
    &\quad \lesssim o(1)( \|\delta u\|_{\widetilde{L}^{2}_{t}(\dot{B}^{\frac{d}{2}-1}_{2,1})}+\|\delta u\|_{L^1_{t}(\dot{B}^{\frac{d}{2}}_{2,1})}).
\end{aligned}
\end{equation}
Then, from \eqref{deltae1}-\eqref{deltae11} and \eqref{inter} we get
\begin{equation}\label{deltae12}
\begin{aligned}
\|\delta u\|_{\widetilde{L}^{\infty}_{t}(\dot{B}^{\frac{d}{2}-2}_{2,1})}+\|\delta u\|_{L^1_{t}(\dot{B}^{\frac{d}{2}}_{2,1})}+\|\delta u\|_{\widetilde{L}^2_{t}(\dot{B}^{\frac{d}{2}-1}_{2,1})}&\lesssim \varepsilon.
\end{aligned}
\end{equation}
In addition, in view of \eqref{maximal2}, we have the $L^2$-in-time estimate of $\eqref{diff1}_{1}$ as follows
\begin{equation}\label{deltae2}
\begin{aligned}
&\|\delta u\|_{\widetilde{L}^{\infty}_{t}(\dot{B}^{\frac{d}{2}-1}_{2,1})}^{\ell}+\|\delta u\|_{\widetilde{L}^2_{t}(\dot{B}^{\frac{d}{2}}_{2,1})}^{\ell}\\
&\quad\lesssim \|u_{0}-u_{0}^{*}\|_{\dot{B}^{\frac{d}{2}-1}_{2,1}}^{\ell}+\sum_{i=1}^{d}\|\frac{\partial}{\partial x_{i}} Z_{i}\|_{\widetilde{L}^1_{t}(\dot{B}^{\frac{d}{2}-2}_{2,1})}^{\ell}+\sum_{i=1}^{d}\|\frac{\partial}{\partial x_{i}}(f_{i}(u)-f_{i}(u^*))\|_{\widetilde{L}^2_{t}(\dot{B}^{\frac{d}{2}-2}_{2,1})}^{\ell}
\\ &
\quad\lesssim  \var+\sum_{i=1}^{d}\| Z_{i}\|_{L^1_{t}(\dot{B}^{\frac{d}{2}-1}_{2,1})}+\| \delta u \|_{\widetilde{L}^{2}_{t}(\dot{B}^{\frac{d}{2}-1}_{2,1})}\|g(u)\|_{\widetilde{L}^{\infty}_{t}(\dot{B}^{\frac{d}{2}}_{2,1})}\\
&\quad\quad+\| u^* -\bar{u}\|_{\widetilde{L}^{\infty}_{t}(\dot{B}^{\frac{d}{2}}_{2,1})}\| g(u)-g(u^{*})\|_{\widetilde{L}^{2}_{t}(\dot{B}^{\frac{d}{2}-1}_{2,1})}\\&
\quad\lesssim \var.
\end{aligned}
\end{equation}
where one has used the bound \eqref{deltae12}. For the high-frequency norms we deduce directly from \eqref{Xt} and \eqref{Xtlimit} that
\begin{equation}\label{highdd}
\left\{
\begin{aligned}
&\|\delta u\|_{\widetilde{L}^{\infty}_{t}(\dot{B}^{\frac{d}{2}-1}_{2,1})}^{h}\lesssim \var\|(u-\bar{u},u^*-\bar{u})\|_{L^\infty_t(\dot{B}^{\frac{d}{2}}_{2,1})}^{h}\lesssim \var, \\
&\|\delta u\|_{\widetilde{L}^2_{t}(\dot{B}^{\frac{d}{2}}_{2,1})}^{h}\lesssim  \var \big(\|u-\bar{u}\|_{\widetilde{L}^{\infty}_t(\dot{B}^{\frac{d}{2}}_{2,1})}^{h}\big)^{\frac{1}{2}}\big( \frac{1}{\var^2}\|u-\bar{u}\|_{\widetilde{L}^1_t(\dot{B}^{\frac{d}{2}}_{2,1})}^{h}\big)^{\frac{1}{2}}+\var \|u^*-\bar{u}\|_{L^2_t(\dot{B}^{\frac{d}{2}+1}_{2,1})}^{h}\lesssim \var.
\end{aligned}
\right.
\end{equation}
Finally, $\delta v$ can be estimated by
\begin{equation}\label{deltae3}
    \begin{aligned}
    \|\delta v\|_{L^1_{t}(\dot{B}^{\frac{d}{2}-1}_{2,1})}&\lesssim \|\delta u\|_{L^1_{t}(\dot{B}^{\frac{d}{2}}_{2,1})}+\|f(u)-f(u^*)\|_{L^1_{t}(\dot{B}^{\frac{d}{2}-1}_{2,1})}+\|z\|_{L^1_{t}(\dot{B}^{\frac{d}{2}-1}_{2,1})}\\
    &\lesssim \|\delta u\|_{L^1_{t}(\dot{B}^{\frac{d}{2}}_{2,1})}+\|z\|_{L^1_{t}(\dot{B}^{\frac{d}{2}-1}_{2,1})}\lesssim \var.
    \end{aligned}
\end{equation}
Combining \eqref{deltae1}-\eqref{deltae3} together one obtains \eqref{rate} and thus the proof of Theorem \ref{Thm3} is complete.

\section{The one-dimensional case}\label{section5}

The purpose of this section is to adapt our computations to the one-dimensional case in a simple setting. We consider the one-dimensional Jin-Xin system 
\begin{equation} \label{JXD1d}
\left\{
    \begin{aligned}
&\frac{\partial} {\partial t}u+ \frac{\partial} {\partial x} v =0, \\
&\varepsilon^2\frac{\partial} {\partial t} v+\frac{\partial} {\partial x} u=-\big(v-f(u) \big),
    \end{aligned}
    \right.
\end{equation}
in the case that the unknowns $u,v$ are scalar functions and tend to the zero equilibrium state $\bar{u},\bar{v}=0$ for simplicity. The nonlinear function $f(u)$ is smooth in $u$ and satisfies $f(0)=f'(0)=0$. It is expected that as $\var\rightarrow 0$, the solution $(u,v)$ to System \eqref{JXD1d} converges to $(u^{*},v^{*})$ satisfying the equations
\begin{align}
&\frac{\partial} {\partial t}u^{*}-\frac{\partial^2} {\partial x^2} u^{*}-\frac{\partial} {\partial x}f(u^{*})=0,\label{uheat1d}  
\end{align}
and
\begin{align}
&v^{*}=-\frac{\partial} {\partial x} u^{*}+f(u^{*}).\label{Darcy1d}
\end{align}

We have the following uniform global existence and relaxation limit of solutions to System \eqref{JXD} with weaker regularities.

\begin{Thm}\label{Thm1d} Let $\bar{u}=\bar{v}=0$ and $\varepsilon\in(0,1)$, and assume \eqref{AssumptionA}-\eqref{Assumptionf}. The following statements hold:
\begin{itemize}
    \item Let the threshold $J_{\var}$ is given by \eqref{Jvar}. There exists a constant $\eta_2>0$ independent of $\var$ such that for initial data $(u_{0},v_{0})$ satisfying $(u_{0},v_{0})\in\dot{B}^{-\frac{1}{2}}_{2,\infty}\cap\dot{B}^{\frac{1}{2}}_{2,1}$ and
 \begin{equation}\label{small1d}
\|(u_{0},\var v_{0})\|_{\dot{B}^{-\frac{1}{2}}_{2,\infty}\cap\dot{B}^{\frac{1}{2}}_{2,1}}\leq \eta_2,
\end{equation} 
then System \eqref{JXD1d} associated to the initial data $(u_{0},v_{0})$ admits a unique global-in-time solution $(u,v)$ satisfying $(u,v)\in \mathcal{C}_{b}(\mathbb{R}_{+};\dot{B}^{-\frac{1}{2}}_{2,\infty}\cap\dot{B}^{\frac{1}{2}}_{2,1})$ and
\begin{equation}\label{Xt1d}
\begin{aligned}
 &\|(u,\var v)\|_{\widetilde{L}^{\infty}_{t}(\dot{B}^{-\frac{1}{2}}_{2,\infty}\cap\dot{B}^{\frac{1}{2}}_{2,1})}+\|u\|_{\widetilde{L}^1_{t}(\dot{B}^{\frac{3}{2}}_{2,\infty}\cap\dot{B}^{\frac{5}{2}}_{2,1})}^{\ell}+\frac{1}{\var^2}\|u\|_{L^1_{t}(\dot{B}^{\frac{1}{2}}_{2,1})}^{h}\\
   &\quad\quad+\|v\|_{\widetilde{L}^1_{t}(\dot{B}^{\frac{1}{2}}_{2,\infty}\cap \dot{B}^{\frac{3}{2}}_{2,1})}^{\ell}+\frac{1}{\var}\|v\|_{L^1_{t}(\dot{B}^{\frac{1}{2}}_{2,1})}^{h}+\frac{1}{\var}\|z\|_{\widetilde{L}^1_{t}(\dot{B}^{-\frac{1}{2}}_{2,\infty}\cap\dot{B}^{\frac{1}{2}}_{2,1})}^{\ell}+\frac{1}{\var^2}\|z\|_{L^1_{t}(\dot{B}^{-\frac{1}{2}}_{2,1})}^{h}\\
   &\quad \leq C\|(u_{0},\var v_{0})\|_{\dot{B}^{-\frac{1}{2}}_{2,\infty}\cap\dot{B}^{\frac{1}{2}}_{2,1}}\quad\hbox{for all } t\geq 0,
\end{aligned}
\end{equation}
where $z:=\frac{\partial} {\partial x} u+v-f(u)$, and $C>0$ is a constant independent of $\var$ and time.

    \item There exists a constant $\eta_1>0$ such that if $u_0^*\in \dot{B}^{-\frac{1}{2}}_{2,\infty}\cap\dot{B}^{\frac{1}{2}}_{2,1}$ and
\begin{equation}\label{smalllimit1d}
\|u^*\|_{\dot{B}^{-\frac {1}{2}}_{2,\infty}\cap\dot{B}^{\frac{1}{2}}_{2,1}}\leq \eta_{1},
\end{equation}
then System \eqref{uheat1d} associated to the initial data $u^*_0$
admit a unique global solution  $u^*\in \cC_b(\R^+;\dot{B}^{-\frac{1}{2}}_{2,\infty}\cap\dot{B}^{\frac{1}{2}}_{2,1})$. Let  $v^{*}$ be given by \eqref{Darcy1d}, then $(u^*,v^{*})$ satisfies
\begin{equation}\label{Xtlimit1d}
    \begin{aligned}
    &\|u^*\|_{\widetilde{L}^{\infty}_{t}(\dot{B}^{-\frac{1}{2}}_{2,\infty}\cap\dot{B}^{\frac{1}{2}}_{2,1})}+\|u^*\|_{\widetilde{L}^1_{t}(\dot{B}^{\frac{3}{2}}_{2,\infty})}+\|v^*\|_{\widetilde{L}^1_{t}(\dot{B}^{\frac{1}{2}}_{2,\infty})}\leq C\|u^*\|_{\dot{B}^{-\frac{1}{2}}_{2,\infty}\cap\dot{B}^{\frac{1}{2}}_{2,1}}\quad \hbox{for all } t\geq 0,
    \end{aligned}
\end{equation}
with $C>0$ a constant independent of time.

    \item Assume further that
\begin{align}\label{errorinitial1d}
    \|u_0-u_0^*\|_{B^{-\frac{3}{2}}_{2,\infty}\cap\dot{B}^{-\frac{1}{2}}_{2,\infty}}\leq\varepsilon.
\end{align}
Then for all $t>0$, the following convergence estimates hold:
\begin{equation}\label{rate1d}
\begin{aligned}
\|u-u^*\|_{\widetilde{L}^{\infty}_{t}(\dot{B}^{-\frac{3}{2}}_{2,\infty}\cap\dot{B}^{-\frac{1}{2}}_{2,\infty})}&+\|u-u^{*}\|_{\widetilde{L}^1_{t}(\dot{B}^{\frac{1}{2}}_{2,\infty})\cap\widetilde{L}^2_{t}(\dot{B}^{\frac{1}{2}}_{2,\infty})}+\|v-v^{*}\|_{\widetilde{L}^1_{t}(\dot{B}^{-\frac{1}{2}}_{2,\infty})}\leq C\varepsilon,
\end{aligned}
\end{equation}
where $C>0$ is a constant independent of $\var$ and time.

\end{itemize}

\end{Thm}
\begin{Rmk}
Compared to the 1-d analysis done in \cite{CBD1} for the compressible Euler system with damping, here we need to work in a weaker regularity setting in order to control the quadratic term $f$. This leads to the consideration of the third index $r=\infty$ so as to be able to use product laws. Indeed, one may not apply such laws in $\dot{B}^{-\frac{1}{2}}_{2,1}$ but it is possible to do so in $\dot{B}^{-\frac{1}{2}}_{2,\infty}$.
\end{Rmk}
\vspace{2ex}

\noindent
\textbf{Proof of Theorem \ref{Thm1d}.} Let $(u,v)$ be a solution to System \eqref{JXD1d} subject to the initial data $(u_{0},v_{0})$. We show the uniform a-priori estimates of $(u,v)$. Arguing similarly as in \eqref{ud21}-\eqref{lowd211}, we have
\begin{equation}
\begin{aligned}
&\|u\|_{\widetilde{L}^{\infty}_{t}(\dot{B}^{-\frac{1}{2}}_{2,\infty})}^{\ell}+ \|u\|_{\widetilde{L}^1_{t}(\dot{B}^{\frac{3}{2}}_{2,\infty})}^{\ell}+\var\|z\|_{\widetilde{L}^{\infty}_{t}(\dot{B}^{-\frac{1}{2}}_{2,\infty})}^{\ell}+\frac{1}{\varepsilon}\|z\|_{\widetilde{L}^1_{t}(\dot{B}^{-\frac{1}{2}}_{2,\infty})}^{\ell}\\
&\quad\lesssim \|(u_{0},\var v_{0})\|_{\dot{B}^{-\frac{1}{2}}_{2,\infty}}+\|f(u)\|_{\widetilde{L}^1_{t}(\dot{B}^{\frac{1}{2}}_{2,\infty})}^{\ell}+\var \|f’(u)\frac{\partial}{\partial x}v\|_{\widetilde{L}^1_{t}(\dot{B}^{-\frac{1}{2}}_{2,\infty})}^{\ell}.
\end{aligned}
\end{equation}
Thence \eqref{1d1d} and the composition estimate \eqref{q5} imply that
\begin{equation}\label{1dn1}
\begin{aligned}
\|f(u)\|_{\widetilde{L}^1_{t}(\dot{B}^{\frac{1}{2}}_{2,\infty})}^{\ell}&\lesssim \|u\|_{L^2_{t}(L^{\infty})}(\|u\|_{\widetilde{L}^2_{t}(\dot{B}^{\frac{1}{2}}_{2,\infty})}^{\ell}+\|u\|_{\widetilde{L}^2_{t}(\dot{B}^{\frac{1}{2}}_{2,1})}^{h})\\
&\lesssim (\|u\|_{\widetilde{L}^2_{t}(\dot{B}^{\frac{1}{2}}_{2,1})}^{\ell}+\frac{1}{\var}\|u\|_{\widetilde{L}^2_{t}(\dot{B}^{\frac{1}{2}}_{2,1})}^{h})^2.
\end{aligned}
\end{equation}
The key is to make use of the product law \eqref{uv3} in the weaker Besov space $B^{-\frac{1}{2}}_{2,\infty}$:
\begin{equation}\label{1dn2}
\begin{aligned}
\var\|f’(u)\frac{\partial}{\partial x}v\|_{\widetilde{L}^1_{t}(\dot{B}^{-\frac{1}{2}}_{2,\infty})}^{\ell}&\lesssim \|f’(u)\|_{\widetilde{L}^{\infty}_{t}(\dot{B}^{\frac{1}{2}}_{2,1})}\|\frac{\partial}{\partial x}v\|_{\widetilde{L}^1_{t}(\dot{B}^{-\frac{1}{2}}_{2,\infty})}\\
&\lesssim \|u\|_{\widetilde{L}^{\infty}_{t}(\dot{B}^{\frac{1}{2}}_{2,1})}(\|v\|_{\widetilde{L}^1_{t}(\dot{B}^{\frac{1}{2}}_{2,\infty})}^{\ell}+\|v\|_{\widetilde{L}^1_{t}(\dot{B}^{\frac{1}{2}}_{2,1})}^{h}).
\end{aligned}
\end{equation}
As in \eqref{lowd2150}-\eqref{lowd2152}, we get
\begin{equation}\nonumber
\begin{aligned}
&\|(u,\var z)\|_{\widetilde{L}^{\infty}_{t}(\dot{B}^{\frac {1}{2}}_{2,1})}^\ell+\|u\|^\ell_{L^1_{t}(\dot{B}^{\frac{5}{2}}_{2,1})}+\frac{1}{\var}\|z\|^\ell_{L^1_{t}(\dot{B}^{\frac{3}{2}}_{2,1})}\\
&\quad\lesssim \|(u_{0},\var v_{0})\|_{\dot{B}^{\frac{1}{2}}_{2,1}}+\var\|  f(u)\|^\ell_{L^1_{t}(\dot{B}^{\frac{3}{2}}_{2,1})}+\var\|f’(u)\frac{\partial}{\partial x}v\|_{L^1_{t}(\dot{B}^{\frac{1}{2}}_{2,1})}^{\ell}.
\end{aligned}
\end{equation}
According to \eqref{1dn1}-\eqref{1dn2} and the low-frequency property
\begin{equation}\nonumber
\begin{aligned}
 \|g\|_{\dot{B}^{s}_{2,1}}^{\ell}\leq \sum_{j\leq J_{\var}} 2^{js'} \|g\|_{\dot{B}^{s-s'}_{2,\infty}}^{\ell}\lesssim 2^{J_{\var}s'}\|g\|_{\dot{B}^{s-s'}_{2,\infty}}^{\ell},
\end{aligned}
\end{equation}
there holds
\begin{equation}\nonumber
\left\{
\begin{aligned}
&\var \|  f(u)\|^\ell_{L^1_{t}(\dot{B}^{\frac{3}{2}}_{2,1})}\lesssim \|f(u)\|_{\widetilde{L}^1_{t}(\dot{B}^{\frac{1}{2}}_{2,\infty})}\lesssim (\|u\|_{\widetilde{L}^2_{t}(\dot{B}^{\frac{1}{2}}_{2,1})}^{\ell}+\frac{1}{\var}\|u\|_{\widetilde{L}^2_{t}(\dot{B}^{\frac{1}{2}}_{2,1})}^{h})^2,\\
&\var \|  f’(u) \frac{\partial}{\partial x}v\|^\ell_{L^1_{t}(\dot{B}^{\frac{1}{2}}_{2,1})}\lesssim \|  f’(u) \frac{\partial}{\partial x}v\|^\ell_{L^1_{t}(\dot{B}^{-\frac{1}{2}}_{2,\infty})}\lesssim\|u\|_{\widetilde{L}^{\infty}_{t}(\dot{B}^{\frac{1}{2}}_{2,1})}(\|v\|_{\widetilde{L}^1_{t}(\dot{B}^{\frac{1}{2}}_{2,\infty})}^{\ell}+\|v\|_{\widetilde{L}^1_{t}(\dot{B}^{\frac{1}{2}}_{2,1})}^{h}) .
\end{aligned}
\right.
\end{equation}
Thus, we obtain the low-frequency estimates
\begin{equation}\label{1dlow}
\begin{aligned}
&\|(u,\var v)\|_{\widetilde{L}^{\infty}_{t}(\dot{B}^{-\frac{1}{2}}_{2,\infty}\cap\dot{B}^{\frac{1}{2}}_{2,1})}^{\ell}+\|u\|_{\widetilde{L}^1_{t}(\dot{B}^{\frac{3}{2}}_{2,\infty}\cap\dot{B}^{\frac{5}{2}}_{2,1})}^{\ell}+\|v\|_{\widetilde{L}^1_{t}(\dot{B}^{\frac{1}{2}}_{2,\infty}\cap \dot{B}^{\frac{3}{2}}_{2,1})}^{\ell}+\frac{1}{\var}\|z\|_{\widetilde{L}^1_{t}(\dot{B}^{-\frac{1}{2}}_{2,\infty}\cap\dot{B}^{\frac{1}{2}}_{2,1})}^{\ell}\\
   &\quad\lesssim \|(u_{0},\var v_{0})\|_{\dot{B}^{-\frac{1}{2}}_{2,\infty}\cap\dot{B}^{\frac{1}{2}}_{2,1}}+(\|u\|_{\widetilde{L}^2_{t}(\dot{B}^{\frac{1}{2}}_{2,1})}^{\ell}+\frac{1}{\var}\|u\|_{\widetilde{L}^2_{t}(\dot{B}^{\frac{1}{2}}_{2,1})}^{h})^2\\
   &\quad\quad+\|u\|_{\widetilde{L}^{\infty}_{t}(\dot{B}^{\frac{1}{2}}_{2,1})}(\|v\|_{\widetilde{L}^1_{t}(\dot{B}^{\frac{1}{2}}_{2,\infty})}^{\ell}+\|v\|_{\widetilde{L}^1_{t}(\dot{B}^{\frac{1}{2}}_{2,1})}^{h}) .
\end{aligned}
\end{equation}

On the other hand, one deduces from \eqref{1d1d}, interpolation \eqref{inter} and the composition estimate \eqref{q2} that
\begin{equation}\nonumber
\begin{aligned}
\frac{1}{\varepsilon }\|f(u)\|_{L^1_t(B^{\frac{1}{2}}_{2,1})}^h&\lesssim \|u\|_{\widetilde{L}^{2}_{t}(\dot{B}^{\frac{1}{2}}_{2,1})}\big{(} \|u\|_{\widetilde{L}^2_{t}(\dot{B}^{\frac{3}{2}}_{2,1})}^{\ell}+\frac{1}{\var}\|u\|_{\widetilde{L}^2_{t}(\dot{B}^{\frac{1}{2}}_{2,1})}^{h}\big)\\
&\lesssim \big{(} \|u\|_{\widetilde{L}^2_{t}(\dot{B}^{\frac{1}{2}}_{2,1}\cap\dot{B}^{\frac{3}{2}}_{2,1})}^{\ell}+\frac{1}{\var}\|u\|_{\widetilde{L}^2_{t}(\dot{B}^{\frac{1}{2}}_{2,1})}^{h}\big)^2.
\end{aligned}
\end{equation}
Thus, similar arguments as in \eqref{JXD1}-\eqref{Highe2} also lead to the high-frequency estimates
\begin{equation}\label{1dhigh}
\begin{aligned}
&\|( u,\varepsilon  v)\|_{\widetilde{L}^{\infty}_{t}(\dot{B}^{\frac{1}{2}}_{2,1})}^{h}+\frac{1}{\varepsilon^2}\|(u,\var v)\|_{L^1_{t}(\dot{B}^{\frac{1}{2}}_{2,1})}^{h}+\frac{1}{\varepsilon}\|u\|_{\widetilde{L}^2_{t}(\dot{B}^{\frac{1}{2}}_{2,1})}^{h}+ \frac{1}{\varepsilon^2}\|z\|_{L^1_t(B^{-\frac{1}{2}}_{2,1})}^h\\
&\quad\lesssim \|( u_0,\varepsilon  v_0)\|_{\dot{B}^{\frac{1}{2}}_{2,1}}^{h}+\big{(} \|u\|_{\widetilde{L}^2_{t}(\dot{B}^{\frac{1}{2}}_{2,1}\cap\dot{B}^{\frac{3}{2}}_{2,1})}^{\ell}+\frac{1}{\var}\|u\|_{\widetilde{L}^2_{t}(\dot{B}^{\frac{1}{2}}_{2,1})}^{h}\big)^2.
\end{aligned}
\end{equation}

Finally, in order to enclose the uniform a-priori estimates, the $\widetilde{L}^2_{t}(\dot{B}^{\frac{1}{2}}_{2,1}\cap\dot{B}^{\frac{3}{2}}_{2,1})$-norm of $u^{\ell}$ is needed to  handle the nonlinear term $f(u)$. Indeed, by virtue of the refined interpolation \eqref{inter}, one has
\begin{equation}\label{1d1d}
\left\{
\begin{aligned}
\|u\|_{\widetilde{L}^2_{t}(\dot{B}^{\frac{1}{2}}_{2,1})}^{\ell}&\lesssim \|u^{\ell}\|_{\widetilde{L}^2_{t}(\dot{B}^{\frac{1}{2}}_{2,1})}+\|u^{h}\|_{\widetilde{L}^2_{t}(\dot{B}^{\frac{1}{2}}_{2,1})}^{\ell}\\
&\lesssim \big( \|u\|_{\widetilde{L}^{\infty}_{t}(\dot{B}^{-\frac{1}{2}}_{2,\infty})}^{\ell}\big)^{\frac{1}{2}}\big( \|u\|_{\widetilde{L}^{1}_{t}(\dot{B}^{\frac{3}{2}}_{2,\infty})}^{\ell}\big)^{\frac{1}{2}}+\|u\|_{\widetilde{L}^2_{t}(\dot{B}^{\frac{1}{2}}_{2,1})}^{h},\\
\|u\|_{\widetilde{L}^2_{t}(\dot{B}^{\frac{3}{2}}_{2,1})}^{\ell}&\lesssim \|u^{\ell}\|_{\widetilde{L}^2_{t}(\dot{B}^{\frac{3}{2}}_{2,1})}+\|u^{h}\|_{\widetilde{L}^2_{t}(\dot{B}^{\frac{3}{2}}_{2,1})}^{\ell}\\
&\lesssim \big( \|u\|_{\widetilde{L}^{\infty}_{t}(\dot{B}^{\frac{1}{2}}_{2,\infty})}^{\ell}\big)^{\frac{1}{2}}\big( \|u\|_{\widetilde{L}^{1}_{t}(\dot{B}^{\frac{5}{2}}_{2,\infty})}^{\ell}\big)^{\frac{1}{2}}+\frac{1}{\var}\|u\|_{\widetilde{L}^2_{t}(\dot{B}^{\frac{1}{2}}_{2,1})}^{h}.
\end{aligned}
\right.
\end{equation}
The combination of \eqref{1dlow}-\eqref{1dhigh} gives rise to the uniform a-priori estimates of $(u,v)$. Therefore, we are able to prove the global existence of solutions to the Cauchy problem. The global existence of solutions of System \eqref{uheat} and the convergence estimates \eqref{rate1d} can be shown similarly as  in the previous sections. The details are omitted.

\section{Appendix}\label{section6}

\vspace{3mm}
\subsection{Some properties of Besov spaces}
This subsection is devoted to recalling some properties of Besov spaces which are used in this paper. 
\begin{Lemma}\cite{bahouri1}
The following statements hold:
\begin{itemize}
\item
The interpolation property is satisfied for $1\leq p\leq\infty$, $s_{1}<s_{2}$ and $\theta\in(0,1)$:
\begin{equation}
\begin{aligned}
&\|u\|_{\dot{B}^{\theta s_{1}+(1-\theta)s_{2}}_{p,1}}\lesssim \frac{1}{\theta(1-\theta)(s_{2}-s_{1})}\|u\|_{\dot{B}^{ s_{1}}_{p,\infty}}^{\theta}\|u\|_{\dot{B}^{s_{2}}_{p,\infty}}^{1-\theta}.\label{inter}
\end{aligned}
\end{equation}
\item{}
Let $s_{1}$, $s_{2}$ and $p$ satisfy $2\leq p\leq \infty$, $s_{1}\leq \frac{d}{p}$, $s_{2}\leq \frac{d}{p}$ and $s_{1}+s_{2}>0$. Then we have
\begin{equation}\label{uv2}
\begin{aligned}
&\|uv\|_{\dot{B}^{s_{1}+s_{2}-\frac{d}{p}}_{p,1}}\lesssim \|u\|_{\dot{B}^{s_{1}}_{p,1}}\|v\|_{\dot{B}^{s_{2}}_{p,1}}.
\end{aligned}
\end{equation}
\item{}
 Assume that $s_{1}$, $s_{2}$ and $p$ satisfy $2\leq p\leq \infty$, $s_{1}\leq \frac{d}{p}$, $s_{2}<\frac{d}{p}$ and $s_{1}+s_{2}\geq0$. Then it holds 
\begin{equation}\label{uv3}
\begin{aligned}
&\|uv\|_{\dot{B}^{s_{1}+s_{2}-\frac{d}{p}}_{p,\infty}}\lesssim \|u\|_{\dot{B}^{s_{1}}_{p,1}}\|v\|_{\dot{B}^{s_{2}}_{p,\infty}}.
\end{aligned}
\end{equation}
\end{itemize}
\end{Lemma}

\begin{Lemma}\cite{bahouri1}\label{Composition}
Let $d\geq1$, $1\leq p,r\leq \infty$, $s\in\mathbb{R}$, and $f(u)$ be smooth in $u$. Then, for every real-valued function $u\in \mathcal{S}(\mathbb{R}^{d})+\bar{u}$ for some constant $\bar{u}>0$, it holds that
\begin{equation}\label{F1}
\begin{aligned}
\norme{f(u)-f(\bar{u})}_{B^{s}_{p,r}}\leq C_{u}\norme{u-\bar{u}}_{B^{s}_{p,r}},
\end{aligned}
\end{equation}
where $C_{u}>0$ denotes a constant dependent of $\|u\|_{L^{\infty}}$, $f''$, $s$, $p$, $r$ and $d$.

Furthermore, if $f'(\bar{u})=0$, then for all $u,v\in\mathcal{S}(\mathbb{R}^{d})+\bar{u}$, it holds that
\begin{equation}\label{F2}
\|f(u)-f(v)\|_{\dot{B}^s_{p,1}} \leq C\Bigl(\|u-v\|_{L^\infty}\|(u-\bar{u},v-\bar{u})\|_{\dot{B}^s_{p,1}} + 
\|u-v\|_{\dot{B}^s_{p,1}} \|(u-\bar{u},v-\bar{u})\|_{L^\infty}\Bigr)\cdotp
\end{equation}
\end{Lemma}

The following lemma can be shown as in \cite{c2} and plays a key role in our nonlinear analysis.

\begin{Lemma}\cite{c2}\label{compositionl2} Let $d\geq1$, $s>0$, $J_{\varepsilon}$ be given by \eqref{Jvar}, and $f$ be any smooth function such that $f'(\bar{u})$ vanishes for some constant $\bar{u}>0$. Then for every real-valued function $u\in \mathcal{S}(\mathbb{R}^{d})+\bar{u}$, there holds
\begin{align}
&\|f(u)-f(\bar{u})\|_{\dot{B}^{s}_{2,1}}^{\ell}\leq C_{u} \|u-\bar{u}\|_{L^{\infty}}\big{(} \|u-\bar{u}\|_{\dot{B}^{s}_{2,1}}^{\ell}+\varepsilon^{\sigma-s}\|u-\bar{u}\|_{\dot{B}^{\sigma}_{2,1}}^{h}\big{)},\quad\quad\sigma\geq0,\label{q1}\\
&\|f(u)-f(\bar{u})\|_{\dot{B}^{s}_{2,1}}^{h}\leq C_{u}\|u-\bar{u}\|_{L^{\infty}}\big{(} \varepsilon^{\sigma-s}\|u-\bar{u}\|_{\dot{B}^{\sigma}_{2,1}}^{\ell}+\|u-\bar{u}\|_{\dot{B}^{s}_{2,1}}^{h}\big{)},\quad\quad \sigma\in\mathbb{R},\label{q2}
\end{align}
where $C_{u}>0$ denotes a constant dependent of $\|u\|_{L^{\infty}}$, $f''$, $s$, $\sigma$ and $d$.

Furthermore, for any $s\geq-\frac{d}{2}$, the following estimate holds{\rm:}
\begin{align}
&\|f(u)-f(\bar{u})\|_{\dot{B}^{s}_{2,\infty}}^{\ell}\leq C_{u}\|u-\bar{u}\|_{\dot{B}^{\frac{d}{2}}_{2,1}}\big{(} \|u-\bar{u}\|_{\dot{B}^{s}_{2,\infty}}^{\ell}+\varepsilon^{\frac{d}{2}-s}\|u-\bar{u}\|_{\dot{B}^{\frac{d}{2}}_{2,1}}^{h}\big{)}.\label{q5}
\end{align}
\end{Lemma}

\subsection{Regularity estimates of linear parabolic and damped equations}
We consider the following Cauchy problem of the linear parabolic equation:
\begin{equation}
\left\{
\begin{aligned}
&\partial_t u- \sum_{i=1}^{d}\frac{\partial}{\partial x_{i}}(A_{i} \frac{\partial}{\partial x_{i}}  u) =F,\quad x\in\mathbb{R}^{d},\quad t>0,\\
&u(0, x)=u_0(x),\quad\quad\quad\quad\quad\quad\quad~~ x\in\mathbb{R}^{d},
\end{aligned}
\right.\label{Heat}
\end{equation}
where the unknown is $u=u(x,t)\in\mathbb{R}^{n}$, and the constant matrices $A_{i}$ are given by \eqref{AssumptionA}.

\begin{Lemma}\label{maximalheat}
Let $d,n\geq1$, $s\in\mathbb{R}$, $1\leq r\leq\infty$, $T>0$ be given time, and assume $u_{0}\in\dot{B}^{s}_{2,r}$.

{\rm(1)} If  $F\in  \widetilde{L}^1(0,T;\dot{B}^{s}_{2,r})$, and $u$ is the solution to the Cauchy problem \eqref{Heat}, then for $t\in(0,T)$, $u$ satisfies
\begin{equation}\label{maximal1}
\begin{aligned}
&\|u\|_{\widetilde{L}^{\infty}_{t}(\dot{B}^{s}_{2,r})}+\|u\|_{\widetilde{L}^1_{t}(\dot{B}^{s+2}_{2,r})}\leq C(\|u_{0}\|_{\dot{B}^{s}_{2,r}}+\|f\|_{\widetilde{L}^1_{t}(\dot{B}^{s}_{2,r})}),
\end{aligned}
\end{equation}
where $C>0$ is a constant independent of $T$.

{\rm(2)} If  $F=F_{1}+F_{2}$ with $F_{1}\in  \widetilde{L}^1(0,T;\dot{B}^{s}_{2,r})$ and $F_{2}\in  \widetilde{L}^2(0,T;\dot{B}^{s-1}_{2,r})$, and $u$ is the solution to the Cauchy problem \eqref{Heat}, then for $t\in(0,T)$, $u$ satisfies
\begin{equation}\label{maximal2}
\begin{aligned}
&\|u\|_{\widetilde{L}^{\infty}_{t}(\dot{B}^{s}_{2,r})}+\|u\|_{\widetilde{L}^2_{t}(\dot{B}^{s+1}_{2,r})}\leq C(\|u_{0}\|_{\dot{B}^{s}_{2,r}}+\|F_{1}\|_{\widetilde{L}^1_{t}(\dot{B}^{s}_{2,r})}+\|F_{2}\|_{\widetilde{L}^2_{t}(\dot{B}^{s-1}_{2,r})}).
\end{aligned}
\end{equation}
\end{Lemma}

\begin{proof}
First, we obtain after taking the $L^2$-inner product of $\eqref{Heat}_{1}$ with $\ddj u$  that
\begin{align}
\frac{d}{dt} \|\ddj u\|_{L^2}^2+\sum_{i=1}^{d}a_{i}\|\frac{\partial}{\partial x_{i}} \ddj u\|^2 \leq \|\ddj u\|_{L^2} \,\|\ddj F\|_{L^2}.\label{871}
\end{align}
Dividing the two sides of \eqref{871} by $(\|\ddj u\|_{L^2}^2+\eta)^{\frac{1}{2}}$, integrating the resulting equation over $[0,t]$, and then taking the limit as $\eta\rightarrow0$, we have
\begin{align*}
\|\ddj u(t)\|_{L^{\infty}_{t}(L^2)}+\min_{1\leq i\leq m}a_{i}2^{2j}\int_0^t\|\ddj u\| d\tau\lesssim \|\ddj u_{0}\|_{L^2}+\int_0^t\|\ddj F\|_{L^2}d\tau,
\end{align*}
which yields
\begin{equation}\nonumber
\begin{aligned}
&\|u\|_{\widetilde{L}^{\infty}_{t}(\dot{B}^{s}_{2,r})}+\min_{1\leq i\leq m}a_{i}\|u\|_{\widetilde{L}^1_{t}(\dot{B}^{s+2}_{2,r})}\lesssim \|u_{0}\|_{\dot{B}^{s}_{2,r}}+\|F\|_{\widetilde{L}^1_{t}(\dot{B}^{s}_{2,r})}.
\end{aligned}
\end{equation}
In accordance with the above estimates, we prove \eqref{maximal1}.

Next, we are going to show \eqref{maximal2}. Integrating \eqref{871} over $[0,t]$, we obtain
\begin{equation}\nonumber
\begin{aligned}
&\|\ddj u\|_{L^{\infty}_{t}(L^2)}^2+\min_{1\leq i\leq m}a_{i} 2^{2j}\int_0^t\|\ddj u\|^2 d\tau\\
&\quad\lesssim \|\ddj u_{0}\|_{L^2}^2+\|\ddj F_{1}\|_{L^2_{t}(L^2)} \|\ddj u\|_{L^{2}_{t}(L^2)}+\|\ddj F_{2}\|_{L^2_{t}(L^2)} \|\ddj u\|_{L^{2}_{t}(L^2)}.
\end{aligned}
\end{equation}
This implies that
\begin{equation}\label{7444}
\begin{aligned}
&\|u\|_{\widetilde{L}^{\infty}_{t}(\dot{B}^{s}_{2,r})}+\min_{1\leq i\leq m}a_{i}\|u\|_{\widetilde{L}^2_{t}(\dot{B}^{s+1}_{2,r})}\\
&\quad\lesssim \|u_{0}\|_{\dot{B}^{s}_{2,r}}+\|u\|_{\widetilde{L}^{\infty}_{t}(\dot{B}^{s}_{2,r})}^{\frac{1}{2}} \|F_{1}\|_{\widetilde{L}^{1}_{t}(\dot{B}^{s}_{2,r})}^{\frac{1}{2}}+\|u\|_{\widetilde{L}^{2}_{t}(\dot{B}^{s+1}_{2,r})}^{\frac{1}{2}} \|F_{2}\|_{\widetilde{L}^{2}_{t}(\dot{B}^{s-1}_{2,r})}^{\frac{1}{2}}.
\end{aligned}
\end{equation}
By \eqref{7444} and the Young inequality, \eqref{maximal2} follows.
\end{proof}

Finally, we consider for the Cauchy problem of the damped equation
\begin{equation}
\left\{
\begin{aligned}
&\partial_t u+a_{0} u =F,\quad\quad~ x\in\mathbb{R}^{d},\quad t>0,\\
&u(0, x)=u_0(x),\quad x\in\mathbb{R}^{d},
\end{aligned}
\right.\label{damped}
\end{equation}
Similarly, we are able to show the following uniform estimates of \eqref{damped}.
\begin{Lemma}\label{maximaldamped}
Let $d,n\geq1$, $s\in\mathbb{R}$, $1\leq r\leq\infty$, $T>0$ be given time, and $a_{0}>0$ be a positive constant. Assume $u_{0}\in\dot{B}^{s}_{2,r}$ and $F\in  \widetilde{L}^1(0,T;\dot{B}^{s}_{2,r})$. If $u$ is the solution to the Cauchy problem \eqref{damped}, then for $t\in(0,T)$, $u$ satisfies
\begin{equation}\label{maximal11}
\begin{aligned}
&\|u\|_{\widetilde{L}^{\infty}_{t}(\dot{B}^{s}_{2,r})}+a_{0}\|u\|_{\widetilde{L}^1_{t}(\dot{B}^{s}_{2,r})}\leq C(\|u_{0}\|_{\dot{B}^{s}_{2,r}}+\|F\|_{\widetilde{L}^1_{t}(\dot{B}^{s}_{2,r})}),
\end{aligned}
\end{equation}
where $C>0$ is a constant independent of $a_{0}$ and $T$.
\end{Lemma}

\vspace{2ex}
\textbf{Acknowledgments}
TCB is partially  supported by the European Research Council (ERC) under the European Union's Horizon 2020 research and innovation programme (grant agreement NO: 694126-DyCon).  

  \bigbreak
\textbf{Conflicts of interest.}
 On behalf of all authors, the corresponding author states that there is no conflict of interest. 
 
\bigbreak
\textbf{Data availability statement.}
 Data sharing not applicable to this article as no data sets were generated or analysed during the current study.

\end{document}